\documentclass[onefignum,onetabnum]{siamonline181217}

\usepackage{amsmath,amsfonts,amssymb}
\usepackage{graphicx,bm,mathdots,multirow,subcaption,autonum}
\ifpdf
\DeclareGraphicsExtensions{.eps,.pdf,.png,.jpg}
\else
\DeclareGraphicsExtensions{.eps}
\fi

\usepackage{enumitem}
\setlist[enumerate]{leftmargin=.5in}
\setlist[itemize]{leftmargin=.5in}

\newsiamremark{remark}{Remark}
\newsiamremark{assumption}{Assumption}
\newsiamremark{example}{Example}

\newcommand{\rev}[1]{\textcolor{black}{#1}}

\headers{Block preconditioning in stochastic Galerkin}{M. Kub\'{i}nov\'{a}, I. Pultarov\'{a} }

\title{Block preconditioning of stochastic Galerkin problems:\\
	New two-sided guaranteed spectral bounds\thanks{This version dated August 20, 2019.
		\funding{The work of M.\,K. was supported by the Czech Academy of Sciences through the project L100861901 (Programme for promising human resources -- postdocs) and by the Ministry of Education, Youth and Sports of the Czech Republic through the project LQ1602 (IT4Innovations excellence in science). The work of I.\,P. was supported by the Grant Agency of the Czech Republic under the contract No.~17-04150J.}}}

\author{Marie Kub\'{i}nov\'{a}\thanks{Institute of Geonics of the CAS, Ostrava, Czech Republic
		(\email{marie.kubinova@ugn.cas.cz}, \href{http://www.ugn.cas.cz/\%7Ekubinova/}{http://www.ugn.cas.cz/\textasciitilde kubinova}).}
	\and Ivana Pultarov\'{a}\thanks{Faculty of Civil Engineering, Czech Technical University, Prague, Czech Republic,
		\rev{and College of Polytechnics Jihlava, Czech Republic}
		(\email{ivana.pultarova@cvut.cz}, \url{http://mat.fsv.cvut.cz/ivana}).}}

\usepackage{amsopn}

\begin{document}

\maketitle

\begin{abstract}
  The paper focuses on numerical solution of parametrized \rev{diffusion} equations with scalar parameter-dependent coefficient function by the stochastic (spectral) Galerkin method. We study preconditioning of the related discretized problems using preconditioners obtained by modifying the stochastic part of the partial differential equation. We present a simple but general approach for obtaining two-sided bounds to the spectrum of the resulting  matrices, based on a particular splitting of the discretized operator. Using this tool and considering the stochastic approximation space formed by classical orthogonal polynomials, we obtain new spectral 
  bounds depending solely on the properties of the coefficient function and 
  the type of the approximation polynomials for several classes of block-diagonal preconditioners.
  These bounds are guaranteed and applicable to various distributions of parameters. Moreover, the conditions on the parameter-dependent coefficient function are only local, and therefore less restrictive than those usually assumed in the literatur\rev{e.}
\end{abstract}



\section{Introduction}\label{sec:intro}

Growing interest in uncertainty quantification of numerical solutions of partial differential equations stimulates new modifications of standard numerical methods. A popular choice for partial differential equations with parametrized or uncertain data is the stochastic Galerkin method \cite{Babuska2004Galerkin,XiuBook}. Similarly to deterministic problems, approximate solutions, which depend on physical and stochastic variables (parameters), are searched for in finite-dimensional subspaces of the original Hilbert space. More precisely, the approximate solutions are orthogonal projections of the exact solution to the finite-dimensional subspaces with respect to the energy inner product defined by the operator of the equation; see, e.g.,~\cite{Bespalov2014Energy,Eigel2016Local,EigelPfeifer2017,Powell2009Block}. The approximation subspaces are considered in the form of a tensor product of a physical variable space (finite-element functions) and a stochastic variable space (polynomials); see, e.g.,~\cite{Babuska2004Galerkin,Ernst2010Stochastic}. The form and  qualities of the system matrix $\bm A$ of the discretized problem are determined by the structure of the uncertain data and the type of the finite-dimensional solution spaces. For special classes of parameters, it was shown, see, e.g., \cite{Powell2009Block,Sousedik2014Hierarchical}, that certain block-diagonal matrices are spectrally equivalent to $\bm A$ independently of the degree of polynomials and the number of random parameters, and thus they can be used for preconditioning. Having a good preconditioning method or, in other words, a good and feasible approximation of $\bm{A}^{-1}$, we may also efficiently estimate a~posteriori the energy norm of the error during iterative solution processes~\cite{AinsworthOden,Bespalov2014Energy,Crowder2018,
Eigel2016Local,Khan2018Robust}. This estimate can be used in adaptive algorithms~\cite{Bespalov2014Energy,Crowder2018Efficient,Eigel2014Adaptive}. In practice, matrix $\bm A$ is never built explicitly, only matrix-vector products are 
evaluated (\cite{PowellSilvesterSim}).

In this paper, we focus on matrices arising in the discretized stochastic Galerkin method and present new guaranteed two-sided bounds to the spectra of the preconditioned matrices for several types of preconditioner. We consider only preconditioning with respect to the stochastic parts of problems, and thus we assume that a suitable preconditioning method or an efficient solver for the underlying deterministic problem is available; see, e.g.,~\cite{ElmanFurnival,Mueller2018Bramble,Subber2014Schwarz}. We formulate an idea of obtaining bounds to the spectra of the preconditioned matrix from the spectrum of small Gram matrices depending solely on the stochastic part of the approximation space. The motivation, however, comes from techniques and tools of the algebraic multilevel preconditioning introduced in~\cite{Eijkhout1991role,Axelsson1996Iterative}. Similar idea was, in a simpler form, used already
in~\cite{Pultarova2016Hierarchical,Pultarova2017BLock}. In the current paper, it is applied in a more general setting, and we believe that the derived technique may lead to an improvement of some other recently introduced estimates, such as \cite{Khan2018Robust,Mueller2018Bramble}. The derived technique is also applicable to systems in the form of multi-term matrix equation (see~\cite[eq.~(1.8)]{PowellSilvesterSim}).

The paper is organized as follows. In \cref{sec:stoch_gal_mat}, we briefly recall the stochastic Galerkin method and the structure of the  matrices $\bm A$ of the resulting systems of linear equations for the tensor product polynomials and complete polynomials. Since the structure of $\bm A$ plays a crucial role in the analysis, theoretical considerations will be accompanied by illustrative examples throughout the paper. \Cref{sec:sp_eq} formulates a general concept of proving spectral equivalence for a broad class of (not only) stochastic Galerkin preconditioners.  In \cref{sec:precond}, we apply this idea to preconditioners which are represented by a special type of block-diagonal (or Schur complement) approximations of $\bm A$, and show how to obtain the spectral bounds of the preconditioned problems from the spectral bounds of small Gram matrices of the corresponding polynomial chaos. We also evaluate those bounds explicitly for the considered polynomial chaoses. Simple numerical examples demonstrating the obtained theoretical outcomes are presented at the end of the section. 

Throughout the paper, we denote by $\kappa(\bm{M}^{-1}\bm{A})$, where $\bm{A}$ and $\bm{M}$ are symmetric positive definite, the spectral condition number of $\bm{M}^{-1}\bm{A}$, i.e., the standard condition number of $\bm{M}^{-\frac{1}{2}}\bm{A}\bm{M}^{-\frac{1}{2}}$ or, in other words, $\lambda_\text{max}(\bm{M}^{-1}\bm{A})/\lambda_\text{min}(\bm{M}^{-1}\bm{A})$. By $e_i$ we denote the $i$-th column of the identity matrix, where its size follows from the context.

\section{Stochastic Galerkin matrices}\label{sec:stoch_gal_mat}

Consider the variational problem of finding $u\in\widetilde{V}=H_0^1(D)\otimes L^2_\rho(\Gamma)$, such that
\begin{equation}\label{eq:weak}
\int_\Gamma\int_D a(\bm{x},\bm{\xi})\nabla u(\bm{x},\bm{\xi})\cdot\nabla v(\bm{x},\bm{\xi})\rho(\bm{\xi})\,{\rm d}\bm{x}\,{\rm d}\bm{\xi}=\int_\Gamma\int_D f(\bm{x})v(\bm{x},\bm{\xi})\rho(\bm{\xi})\,{\rm d}\bm{x}\,{\rm d}\bm{\xi}\quad \text{for all}\; v\in \widetilde{V}
\end{equation}
where $D\subset{\mathbb R}^d$ is a bounded polygonal domain, $d=1,2$ or $3$,
$L^2_\rho(\Gamma)$ is a parametric measure space,
$\Gamma\subset {\mathbb R}^K$, $\Gamma=\prod_{k=1}^K\Gamma_k$, $a\in L^\infty(D)\otimes L^\infty_\rho(\Gamma)$, and $f\in L^2(D)$. 
The gradient is applied only with respect to the 
(physical) variable $\bm{x}\in D$. Let $\bm{\xi}=(\xi_1,\dots,\xi_K)\in \Gamma$, where $\xi_k\in \Gamma_k$ are outcomes of independent random variables with probability densities $\rho_k(\xi_{k})$, $k=1,\dots,K$. The joint probability density is then $\rho=\prod_{k=1}^K\rho_k$. In the following, we consider $\rho_k$ defined on $\mathbb R$ such that $\rho_k(\xi_k)=0$ outside $\Gamma_k$. Thus, instead of $\Gamma_k$ and $\Gamma$, we further write $\mathbb R$ and $\mathbb{R}^K$, respectively. For the convenience of notation, the probability densities are not normalized, see also~\cref{tab:Wiener-Askey}, and we further refer to them as weights.

We assume $a(\bm{x},\bm{\xi})$ in the affine form
\begin{equation}\label{eq:a_expansion}
a(\bm{x},\bm{\xi})=a_0(\bm{x})+\sum_{k=1}^K a_k(\bm{x})\xi_k,
\end{equation}
where $a_k\in L^\infty(D)$, $k=1,\dots,K$.
While it is usually assumed that there exist constants $\underline{a}$ 
and $\overline{a}$ such that
\begin{equation}\label{alfa12}
0<\underline{a}\le a(\bm{x},\bm{\xi})\le \overline{a}<\infty\quad\text{for a.a.}\; \bm{x}\in D,\;
\bm{\xi}\in \Gamma,
\end{equation}
in this paper we consider more general functions $a$. We will only require that the left-hand side of~\cref{eq:weak} defines an inner product on a finite-dimensional approximation space $V\subset\widetilde{V}$; see \cref{sec:pos-def}. This will allow us to use random variables $\xi_k$ with unbounded images and still obtain positive definite system matrices. In other words, we can avoid truncation of supports of distribution functions or any other modification of them. Of course, under such (weaker) condition on $a$, \cref{eq:weak} may not
be well-defined. In this paper we, however, focus only on the discretized problem obtained from~\cref{eq:weak}; see also the discussion in \cite{Powell2009Block}.

We consider discretization using the tensor product
space~\cite{Babuska2010stochastic,Babuska2004Galerkin,Ernst2010Stochastic} of the form $ V=V^{\rm FE}\otimes P\subset\widetilde V$, where $V^{\rm FE}\subset H_0^1(D)$ is 
an $N_{\rm FE}$-dimensional space spanned by the finite-element (FE) functions $\phi_1,\dots,\phi_{N_{\rm FE}}$,
and $P$ is an $N_{\rm P}$-dimensional space spanned by  $K$-variate 
polynomials $\Psi_1$, \dots, $\Psi_{N_{\rm P}}$ 
of variables $\xi_1$, \dots, $\xi_K$.
Denoting the basis functions $\phi_r\Psi_j$ of $V$
by a couple of coordinates $r=1,\dots,N_{\rm FE}$ and $j=1,\dots, N_{\rm P}$, 
we obtain the matrix $\bm A$ of the system of linear equations of the discretized Galerkin problem~\cref{eq:weak} with elements
\begin{eqnarray}
A_{ri,sj}&=&\int_{{\mathbb R}^K}\int_D a(\bm{x},\bm{\xi})\nabla \phi_s(\bm{x})\cdot\nabla \phi_r(\bm{x})\Psi_j(\bm{\xi})\Psi_i(\bm{\xi})\rho(\bm{\xi})
\,{\rm d}\bm{x}\,{\rm d}\bm{\xi}\nonumber\\
&=&\int_D a_0(\bm{x})\nabla \phi_s(\bm{x})\cdot\nabla \phi_r(\bm{x})\,{\rm d}\bm{x}\int_{{\mathbb R}^K}\Psi_j(\bm{\xi})\Psi_i(\bm{\xi})\rho(\bm{\xi})
\,{\rm d}\bm{\xi}\nonumber\\
&&+\sum_{k=1}^K\int_D  a_k(\bm{x})\nabla \phi_s(\bm{x})\cdot\nabla \phi_r(\bm{x})\,{\rm d}\bm{x}\int_{{\mathbb R}^K}\xi_k\Psi_j(\bm{\xi})\Psi_i(\bm{\xi})\rho(\bm{\xi})
\,{\rm d}\bm{\xi}\nonumber\\
&=:&(\bm{F}_0)_{rs}(\bm{G}_0)_{ij}+\sum_{k=1}^K(\bm{F}_k)_{rs}(\bm{G}_k)_{ij},\nonumber
\end{eqnarray}
where for $k=0,1,\dots,K$
\begin{equation}\label{eq:FGelem}
(\bm{F}_k)_{rs}=\int_D  a_k(\bm{x})\nabla \phi_s(\bm{x})\cdot\nabla \phi_r(\bm{x})\,{\rm d}\bm{x}\quad\text{and}\quad
(\bm{G}_k)_{ij}=\int_{{\mathbb R}^K}\xi_k\Psi_j(\bm{\xi})\Psi_i(\bm{\xi})\rho(\bm{\xi})
\,{\rm d}\bm{\xi},
\end{equation}
where we formally set $\xi_0=1$. If the numbering of the basis functions 
$\phi_r\Psi_i$ is anti-lex\-i\-co\-graph\-i\-cal, the structure of $\bm A$ is
\begin{equation}\label{eq:matrixA}
\bm{A}=\sum_{k=0}^K\bm{G}_k\otimes \bm{F}_k.
\end{equation}
In other words, the matrix $\bm A$ is composed of $N_{\rm P}\times N_{\rm P}$ blocks, each of size $N_{\rm FE}\times N_{\rm FE}$.

\begin{example}\label{ex:1}
	Assume $K=1$, the uniform distribution $\rho(\bm{\xi}) = \rho(\xi_1) =\chi_{[ -1,1]}$, and let $\Psi_0(\xi_1)$, $\Psi_1(\xi_1)$ and $\Psi_2(\xi_1)$ be the normalized Legendre orthogonal polynomials of degrees 0, 1, and 2, see \cref{tab:Wiener-Askey}. Then $N_{\rm P}=3$ and
	\begin{equation}\label{eq:ex1}
	\bm{A}=\begin{pmatrix}\bm{F}_0 &\frac{1}{\sqrt{3}}\bm{F}_1 & 0\\
	\frac{1}{\sqrt{3}}\bm{F}_1& \bm{F}_0&\frac{2}{\sqrt{15}}\bm{F}_1 \\
	0&\frac{2}{\sqrt{15}}\bm{F}_1  &\bm{F}_0
	\end{pmatrix}= \bm{I}\otimes\bm{F_0} + \bm{G_1}\otimes \bm{F_1}.
	\end{equation}
\end{example}

\subsection{Approximation spaces and their bases}\label{sec:approx_spaces}

For approximation of the physical part of the solution, we use an $N_{\rm FE}$-dimensional space $V^{\rm FE}$. To approximate the stochastic part of the solution we use the $N_{\rm P}$-dimensional space $P$ of $K$-variate polynomials $\Psi_j(\bm{\xi})=\prod_{k=1}^K\psi_{j_k}^{(k)}(\xi_k)$, $j=1,\dots,N_{\rm P}$. To simplify the notation, we assume that the parameters $\xi_1$, \dots, $\xi_K$ are identically
distributed, i.e., $\rho_1=\dots=\rho_K$. Thus, we omit the superscripts and subscripts $k$ in $\psi_j^{(k)}$ and $\rho_k$, respectively. The extension of the results to polynomial bases with different $\rho_1,\ldots,\rho_K$ is straightforward.

In practice, sets of {\it complete polynomials} (C) or 
{\it tensor product polynomials} (TP)
are usually used; see, e.g., \cite{Ernst2010Stochastic,Powell2009Block}. 
The set of the tensor product polynomials of the degree at most $s_k-1$ in variable $\xi_k$, $k=1,\dots,K$,
is defined as 
\begin{equation}\nonumber
P^{\rm TP}_{s_1,\dots,s_K}=\{ p(\bm{\xi})=\prod_{k=1}^K p_{k}(\xi_k);\; \text{deg}
\,(p_k)\le s_k-1, \; k=1,\dots,K\}\quad\text{and}\quad N_{\rm P}=\prod_{k=1}^Ks_k.
\end{equation}
Let us denote by $V^{\rm TP}_{s_1,\dots,s_K}=
V^{\rm FE}\otimes P^{\rm TP}_{s_1,\dots,s_K}$ the corresponding approximation 
space of~\cref{eq:weak}.
The set of complete polynomials of the maximum total degree $s-1$ is defined as
\begin{equation}\nonumber
P^{\rm C}_s=\{ p(\bm{\xi})=\prod_{k=1}^K p_{k}(\xi_k);\; \sum_{k=1}^K\text{deg}
\,(p_k)\le s-1\}\quad\text{and}\quad N_{\rm P}= \binom{K+s-1}{K}.
\end{equation}
Let us denote by $V^{\rm C}_{s}=V^{\rm FE}\otimes P^{\rm C}_s$ the corresponding approximation space of~\cref{eq:weak}.

For both $P^{\rm TP}_{s_1,\dots,s_K}$ and $P^{\rm C}_{s}$, the bases 
are usually constructed as products of $K$ \emph{classical orthogonal polynomials}. More precisely $\Psi_j(\bm{\xi})=\prod_{k=1}^K\psi_{j_k}(\xi_k)$, $j=1,\dots,N_{\rm P}$, where $\psi_i$ are normalized orthogonal polynomials of the degrees $i=0,1,\dots$, with respect to the weight function $\rho$, i.e.,
\begin{equation}
\int_{\mathbb R}\psi_{i}(\xi)\psi_{j}(\xi)\rho(\xi)\,{\rm d}\xi =\delta_{ij}.
\end{equation}
The $N_{\rm FE}\cdot N_{\rm P}$ basis functions of the discretization space $V$ are then of the form
\begin{equation}\label{eq:order}
\phi_n(\bm{x})\Psi_j(\bm{\xi})=\phi_n(\bm{x})\psi_{j_1}(\xi_1)\dots\psi_{j_K}(\xi_K).
\end{equation}
For the tensor product polynomials, we consider the anti-lexicographical ordering of the basis functions, i.e., the leftmost index ($n$) in \cref{eq:order} is changing the fastest, while the rightmost index ($j_K$) is changing the slowest. For the complete polynomials, we consider ordering by the total degree of the polynomials, going from the smallest to the largest.

Another popular choice of the basis functions of $P$
is a set of \emph{double orthogonal polynomials} \rev{\cite{Babuska2010stochastic,Babuska2004Galerkin,Ernst2010Stochastic}.
If} we use the double orthogonal polynomials as a basis of $P$, 
the matrix $\bm A$ becomes block-diagonal
with the diagonal blocks of the sizes $N_{\rm FE}\times N_{\rm FE}$.
Such  block-diagonal matrix $\bm{A}$ can be also obtained 
by simultaneous diagonalization of all matrices $\bm{G}_k$, see~\cite{Ernst2010Stochastic}.
This diagonal structure of the resulting matrices seems 
favourable for practical computations.
However, the double orthogonal polynomials cannot be used as a basis for
complete polynomials~\cite{Ernst2010Stochastic}. 
Moreover,  for this basis, we cannot obtain methods \rev{for}
a posteriori error 
estimation or adaptivity control in a straightforward way.
In addition, to refine the space $P$, all diagonal blocks of
the matrix $\bm A$
must be \rev{recomputed. Therefore}, in this paper, we only consider 
the classical orthogonal polynomials to construct the bases of $P^{\rm C}$ or
$P^{\rm TP}$.

\subsection{Matrices for classical orthogonal polynomials}
The form of the matrices $\bm{G}_k$, $k = 0,1,\ldots,K$, in~\cref{eq:FGelem} depends on the choice of the basis of  $P^{\rm C}$ or $P^{\rm TP}$ 
and will be important for our future analysis. 
As will be described later, the matrices 
$\bm{G}_k$ can be constructed from (the elements of) a sequence of smaller 
$s\times s$ matrices
\begin{equation}
(\bm{G}_{s,j})_{l+1,m+1}\equiv\int_{\mathbb{R}}\xi^j\psi_l(\xi)
\psi_m(\xi)\rho(\xi)\,{\rm d}\xi, \quad j = 0, 1, \quad l,m = 0,1,\ldots,s-1.
\end{equation}
Let the normalized orthogonal polynomials satisfy the well-known three-term recurrence 
\begin{equation}\label{eq:const_bj_cj}
\sqrt{\beta_{n+1}}\psi_{n+1}(\xi) = (\xi-\alpha_n)\psi_{n}(\xi) - \sqrt{\beta_n}\psi_{n-1}(\xi),\quad n = 1,2,\ldots, \quad \psi_{-1}\equiv 0;
\end{equation} 
then $\bm{G}_{s,0}=\bm{I}_{s}$, where $\bm{I}_{s}$ is the
$s\times s$ identity matrix, and $\bm{G}_{s,1}$ have the form of 
the Jacobi matrix
\begin{equation}\label{eq:G_small}
\bm{G}_{s,1}=\begin{pmatrix}
\alpha_0 & \sqrt{\beta_1}& & \\
\sqrt{\beta_1} & \alpha_1 &\ddots &\\
&\ddots &\ddots&\sqrt{\beta_{s-1}}\\
& & \sqrt{\beta_{s-1}}&\alpha_{s-1}
\end{pmatrix}.
\end{equation}
The eigenvalues of this matrix are given by the roots 
of the polynomial $\psi_{s}$, which are distinct and lie in the support of $\rho$; see, e.g.,~\cite{Gautschi}.  In \cref{tab:Wiener-Askey}, we list the classical  orthogonal polynomials with symmetric statistical distribution considered here together with 
the weight function  $\rho$ corresponding to the non-normalized probability density. Note that due to the symmetry, the diagonal entries of $\bm{G}_{s,1}$ in \cref{eq:G_small} become trivially zero.
These matrices will play a crucial role in deriving spectral bounds, see \cref{sec:precond}. 
\begin{table}[!ht]
	\centering
	{\footnotesize \begin{tabular}{lllllll}
		\hline 
		statistical distribution & weight function  & support& polynomial chaos & $\beta_n$ & $\alpha_n$\\ 
		\hline\hline 
		Gaussian & $e^{-\frac{x^2}{2}}$ &$(-\infty, \infty)$ &  Hermite& $\frac{n}{2}$ & $0$ \\\hline 
		Symmetric Beta & $(1-x^2)^{\gamma-\frac{1}{2}}$ & $[-1,1]$& Gegenbauer  & $\frac{(n+2\gamma -1)n}{(2n-2+2\gamma)(2n+2\gamma)}$& 0 \\ 
		Wigner semicircle & $(1-x^2)^{\frac{1}{2}}$ &  $[-1,1]$& Chebyshev ($2^\text{nd}$ kind) & $\frac{1}{4}$&0 \\ 
		Uniform & $1$ & $[-1,1]$& Legendre  &$\frac{n^2}{(2n-1)(2n+1)}$ & $0$
	\end{tabular}}
	\caption{Wiener--Askey table: symmetric statistical distributions together with the corresponding polynomial chaos (classical orthogonal polynomials) and the three-term recurrence coefficients.}\label{tab:Wiener-Askey}
\end{table}

For the tensor product polynomials, the matrices $\bm{G}_k$, $k=0,1,\dots,K$, are obtained as
\begin{align}\label{eq:AsumGGF}
\bm{G}_0&=\bm{G}_{s_K,0}\otimes \bm{G}_{s_{K-1},0}\otimes \dots \otimes \bm{G}_{2,0}\otimes \bm{G}_{1,0}\\
\bm{G}_1&=\bm{G}_{s_K,0}\otimes \bm{G}_{s_{K-1},0}\otimes \dots \otimes \bm{G}_{2,0}\otimes \bm{G}_{1,1}\\
&\vdots&\\
\bm{G}_K&=\bm{G}_{s_K,1}\otimes \bm{G}_{s_{K-1},0}\otimes \dots \otimes \bm{G}_{2,0}\otimes \bm{G}_{1,0},
\end{align}
see, e.g.,~\cite{Ernst2010Stochastic,Powell2010Preconditioning}. 
\begin{example}\label{ex:2}
	Consider the tensor product Legendre polynomials of two variables $\xi_1$ and $\xi_2$, with $s_1=s_2=3$, then $N_{\rm P}=9$ and the matrix $\bm A$ has the form
	\begin{equation}
	\bm{A}={\footnotesize \left(\begin{array}{ccc|ccc|ccc}
	\bm{F}_0&\frac{1}{\sqrt{3}}\bm{F}_1&0&\frac{1}{\sqrt{3}}\bm{F}_2&0&0&0&0&0\\
	\frac{1}{\sqrt{3}}\bm{F}_1&\bm{F}_0&\frac{2}{\sqrt{15}}\bm{F}_1&0&\frac{1}{\sqrt{3}}\bm{F}_2&0&0&0&0\\
	0&\frac{2}{\sqrt{15}}\bm{F}_1&\bm{F}_0&0&0&\frac{1}{\sqrt{3}}\bm{F}_2&0&0&0\\
	\hline
	\frac{1}{\sqrt{3}}\bm{F}_2&0&0&\bm{F}_0&\frac{1}{\sqrt{3}}\bm{F}_1&0&\frac{2}{\sqrt{15}}\bm{F}_2&0&0\\
	0&\frac{1}{\sqrt{3}}\bm{F}_2&0&\frac{1}{\sqrt{3}}\bm{F}_1&\bm{F}_0&\frac{2}{\sqrt{15}}\bm{F}_1&0&\frac{2}{\sqrt{15}}\bm{F}_2&0\\
	0&0&\frac{1}{\sqrt{3}}\bm{F}_2&0&\frac{2}{\sqrt{15}}\bm{F}_1&\bm{F}_0&0&0&\frac{2}{\sqrt{15}}\bm{F}_2\\
	\hline
	0&0&0&\frac{2}{\sqrt{15}}\bm{F}_2&0&0&\bm{F}_0&\frac{1}{\sqrt{3}}\bm{F}_1&0\\
	0&0&0&0&\frac{2}{\sqrt{15}}\bm{F}_2&0&\frac{1}{\sqrt{3}}\bm{F}_1&\bm{F}_0&\frac{2}{\sqrt{15}}\bm{F}_1\\
	0&0&0&0&0&\frac{2}{\sqrt{15}}\bm{F}_2&0&\frac{2}{\sqrt{15}}\bm{F}_1&\bm{F}_0
	\end{array}\right)}=\sum_{k=0}^2\bm{G}_k\otimes\bm{F}_k,\label{eq:ex_TP_A}\end{equation}
	where the blocks corresponding to the changing degree of the
	approximation polynomials of the variable $\xi_2$ are separated graphically.
\end{example}

For complete polynomials, the matrices $\bm{G}_k$ lose the Kronecker product structure, since $P^\text{C}_s$ is not a tensor product space.
However, since $P^\text{C}_s\subset P^\text{TP}_{s,s,\dots,s}$, each matrix $\bm{G}_k$ is permutation-similar to a submatrix of the matrices in \cref{eq:AsumGGF}, \cite[Lemma 3]{Ernst2010Stochastic}.

\begin{example}\label{ex:3}
	Consider the complete Legendre polynomials of two variables $\xi_1$ and $\xi_2$ and $s=3$, then $N_{\rm P}=6$ and the relevant submatrix of the tensor-product matrix \cref {eq:ex_TP_A} is 
	\begin{equation}\label{eq:ex_TP_C_A}
	\bm{A}={\footnotesize \left(\begin{array}{ccc|ccc|ccc}
	\bm{F}_0&\frac{1}{\sqrt{3}}\bm{F}_1&0&\frac{1}{\sqrt{3}}\bm{F}_2&0&&0&&\\
	\frac{1}{\sqrt{3}}\bm{F}_1&\bm{F}_0&\frac{2}{\sqrt{15}}\bm{F}_1&0&\frac{1}{\sqrt{3}}\bm{F}_2&&0&&\\
	0&\frac{2}{\sqrt{15}}\bm{F}_1&\bm{F}_0&0&0&&0&&\\
	\hline
	\frac{1}{\sqrt{3}}\bm{F}_2&0&0&\bm{F}_0&\frac{1}{\sqrt{3}}\bm{F}_1&&\frac{2}{\sqrt{15}}\bm{F}_2&&\\
	0&\frac{1}{\sqrt{3}}\bm{F}_2&0&\frac{1}{\sqrt{3}}\bm{F}_1&\bm{F}_0&&0&&\\
	&&&&&&\\
	\hline
	0&0&0&\frac{2}{\sqrt{15}}\bm{F}_2&0&&\bm{F}_0&&\\
	&&&&&&&\\
	\hspace*{1cm}&\hspace*{1cm}&\hspace*{1cm}&\hspace*{1cm}&\hspace*{1cm}&\hspace*{1cm}&\hspace*{1cm}&\hspace*{1cm}
	\end{array}\right)}.
	\end{equation}
	Reordering the entries by the total degree of the corresponding polynomial, we obtain
	\begin{equation}
	\bm{A}={\footnotesize \left(\begin{array}{c|cc|ccc}
	\bm{F}_0&\frac{1}{\sqrt{3}}\bm{F}_1&\frac{1}{\sqrt{3}}\bm{F}_2&0&0&0\\
	\hline
	\frac{1}{\sqrt{3}}\bm{F}_1&\bm{F}_0&0&\frac{2}{\sqrt{15}}\bm{F}_1&\frac{1}{\sqrt{3}}\bm{F}_2&0\\
	\frac{1}{\sqrt{3}}\bm{F}_2&0&\bm{F}_0&0&\frac{1}{\sqrt{3}}\bm{F}_1&\frac{2}{\sqrt{15}}\bm{F}_2\\
	\hline
	0&\frac{2}{\sqrt{15}}\bm{F}_1&0&\bm{F}_0&0&0\\
	0&\frac{1}{\sqrt{3}}\bm{F}_2&\frac{1}{\sqrt{3}}\bm{F}_1&0&\bm{F}_0&0\\
	0&0&\frac{2}{\sqrt{15}}\bm{F}_2&0&0&\bm{F}_0\\
	\end{array}\right)}=\sum_{k=0}^2\bm{G}_k\otimes \bm{F}_k,
	\end{equation}
	where the blocks corresponding to the total degrees 0, 1, and 2 are 
	separated graphically. 
\end{example}

\subsection{Positive definiteness}\label{sec:pos-def}
The left-hand side of the equation~\cref{eq:weak} defines the bilinear form
$(\cdot,\cdot)_A$ on $\widetilde V$. We present sufficient conditions on the function $a$, under which $(\cdot,\cdot)_A$ becomes
an inner product (called energy inner product; see, e.g., \cite{Bespalov2014Energy,Eigel2016Local,EigelPfeifer2017}) on the finite-dimensional space $V$. To achieve positive definiteness of the bilinear form $(\cdot,\cdot)_A$, 
we need to assume some dominance of the deterministic part $a_0(\bm{x})$ over the stochastic part $a_k(\bm{x})\xi_k$, $k = 1,\ldots,K$. In this paper, we will assume that there exists a constant $\overline{\mu}\ge 0$ such that
\begin{equation}\label{eq:omega_max}
\sum_{k=1}^K\vert a_k(\bm{x})\vert\le \overline{\mu}\,a_0(\bm{x}), \quad \text{for a.a. }\bm{x}\in D,
\end{equation}
where the particular choice of \,$\overline{\mu}$\, depends on the weight $\rho(\xi)$. For the Beta  distribution on $[ -1,1]$, is suffices to take $\overline{\mu}=1$, while for the Gauss distribution, we take $\overline{\mu}=(2(s_1+\cdots+s_K-K))^{-\frac{1}{2}}$ for tensor product polynomials and $\overline{\mu}=(2s-2)^{-\frac{1}{2}}$ for complete polynomials.\footnote{Since the eigenvalues of matrix $\bm{G}_{s,1}$ are the \rev{zeros} of the Hermite polynomials $\psi_{s}$ and thus lie in the interval $\left\langle -\sqrt{\frac{2(s-1)^2}{s+2}}, \sqrt{\frac{2(s-1)^2}{s+2}} \right\rangle$ \cite[p.120]{Szego1975Orthogonal}, the eigenvalues of $(2s-2)^\frac{1}{2}\,\bm{G}_{s,0}+\bm{G}_{s,1}$ are strictly positive.} Note that this choice of $\overline{\mu}$ also trivially implies that $\bm{G}_{s,0}+\overline{\mu}\,\bm{G}_{s,1}$ is positive definite. For further discussion on bounds of Hermite and \rev{Legendre} 
polynomials see, e.g., \cite{Powell2009Block}.

We emphasize that the assumption~\cref{eq:omega_max} is weaker than the classical assumption
widely used to obtain spectral estimates, e.g.,
\begin{equation}\label{eq:stronger_assum}
\sum_{k=1}^K\Vert a_k(\bm{x})\Vert_{L^\infty(D)}\le \mu_{\rm class}\,\underset{\bm{x}\in D}{\operatorname{ess\, inf}}\, a_0(\bm{x})
\end{equation}
for uniform distribution; see~\cite{Ernst2009Efficient,Khan2018Robust,Mueller2018Bramble,Powell2009Block,Ullmann2010Kronecker}. The main difference between~\cref{eq:omega_max} and \cref{eq:stronger_assum} is that the former is considered point-wise, while the latter uses the norms of $a_k$ over $D$. 
The condition~\cref{eq:omega_max} allows us to obtain not only more accurate two-sided guaranteed bounds to the spectra, but these bounds also apply to parameter distribution and functions $a_k$ for which no estimate could be obtained using the standard approach; see \cref{sec:num_exp}.  
Assumption \cref{eq:omega_max} is sufficient to achieve positive definiteness of $\bm A$. In some applications, we can assume a stronger dominance of $a_0$, i.e.,
\begin{equation}\label{eq:omega}
\sum_{k=1}^K\vert a_k(\bm{x})\vert\le\mu\, a_0(\bm{x}), \quad \text{for a.a. }\bm{x}\in D, \quad 0 \leq \mu\leq\overline{\mu}.
\end{equation}
The smaller the $\mu$, the more favourable spectral bounds of the matrices $\bm{A}$ and of the preconditioned matrices $\bm{M}^{-1}\bm{A}$ are generally achieved. We will further assume that $\mu$ is the smallest number for which \cref{eq:omega} is satisfied.

\section{Proving spectral equivalence of inner products on \texorpdfstring{{\boldmath$V$}}{V}}\label{sec:sp_eq}

We consider preconditioning methods based on inner products that are spectrally equivalent to the energy inner product $(\cdot,\cdot)_A$ on $V$, but are represented by matrices with more favourable non-zero structures such
as, for example, block-diagonal matrices. We base our approach on a
splitting of the inner products to subdomains (\cref{th:lem1}) and on a
preconditioning of a tensor product matrix (\cref{th:lem2}). 

Let $D$ be partitioned into arbitrary non-overlapping elements (subdomains) $\tau_j$, \linebreak[4]
$j = 1,\ldots,N_\text{elem}$.
Consider the following decomposition of $\bm A$ from \cref{eq:matrixA}
\begin{equation}
\bm{A}=\sum_{k=0}^K\bm{G}_k\otimes\bm{F}_k=\sum_{j=1}^{N_{\rm elem}}\sum_{k=0}^K
\bm{G}_k\otimes\bm{F}_k^{(j)} =: \sum_{j=1}^{N_{\rm elem}}\bm{A}^{(j)},
\end{equation}
where
\begin{equation}
(\bm{F}^{(j)}_k)_{rs}=\int_{\tau_j}a_k(\bm{x})\nabla \phi_s(\bm{x})\cdot
\nabla \phi_r(\bm{x})\,{\rm d}\bm{x} \qquad \text{and} \qquad \bm{A}^{(j)}=\sum_{k=0}^K\bm{G}_k\otimes\bm{F}_k^{(j)}.
\end{equation}

\begin{assumption}\label{ass:1}
	We further assume that the functions
	$a_k(\bm{x})$, $k=0,1,\dots,K$, (and thus the function $a(\bm{x},\bm{\xi})$)
	are constant on every element (subdomain) $\tau_j$, $j=1,\dots,N_{\rm elem}$. We define 
	\begin{equation}\label{eq:const_a}
	a_k^{(j)} \equiv a_k(\bm{x}), \quad \bm{x}\in\tau_j,\quad
	j=1,\dots,N_{\rm elem}.
	\end{equation}
	If $a_k(\bm{x})$ are not constant on elements, we would assume a stronger, element-wise, dominance of $a_0(\bm{x})$ over $a_k(\bm{x})$, i.e.,
	\begin{equation}\label{eq:omega_not_const_a}
	\sum_{k=1}^K\underset{\bm{x}\in\tau_j}{\operatorname{ess\, sup}}\,\vert a_k(\bm{x})\vert\le 
	\mu\,\underset{\bm{x}\in\tau_j}{\operatorname{ess\, inf}}\,a_0(\bm{x}),\quad j=1,\dots,N_{\rm elem},
	\end{equation}
	instead of \cref{eq:omega}, which would result in a slight modification of the spectral estimates
	derived in subsequent sections. To simplify the presentation, we do not describe these modifications in more detail.
\end{assumption} 

Using \cref{eq:const_a}, we obtain
\begin{equation}\label{eq:FGelem_i}
(\bm{F}^{(j)}_k)_{rs}=\int_{\tau_j}a_k(\bm{x})\nabla \phi_{s}(\bm{x})\cdot
\nabla \phi_{r}(\bm{x})\,{\rm d}\bm{x} = a_k^{(j)}\int_{\tau_j}\nabla \phi_{s}(\bm{x})\cdot
\nabla \phi_{r}(\bm{x})\,{\rm d}\bm{x} =: a_k^{(j)}(\bm{F}^{(j)})_{rs}.
\end{equation}
Therefore, we can write
\begin{equation}
\bm{A}^{(j)}=\sum_{k=0}^K\bm{G}_k\otimes\bm{F}_k^{(j)}=\sum_{k=0}^K\bm{G}_k\otimes a_k^{(j)}\bm{F}^{(j)} =\left(\sum_{k=0}^Ka_k^{(j)}\bm{G}_k\right)\otimes \bm{F}^{(j)}, 
\end{equation}
which gives
\begin{equation}\label{eq:A_final}
\bm{A}=\sum_{j=1}^{N_{\rm elem}}\bm{A}^{(j)}=\sum_{j=1}^{N_{\rm elem}}\left(\sum_{k=0}^Ka_k^{(j)}\bm{G}_k\right)\otimes \bm{F}^{(j)}.
\end{equation}
In other words, we obtained a decomposition of $\bm A$ in which the dependence of the FE matrices on $a(\bm{x})$ is compensated by splitting of the operator to elements.

In this paper, we consider preconditioners corresponding to an inner product $(\cdot,\cdot)_M$ defined on $V$ whose matrix representation (with respect to the same basis) is of the form analogous to \cref{eq:A_final}, in particular
\begin{align}\label{eq:matrixB}
\bm{M}=\sum_{k=0}^K\widetilde{\bm{G}}_k\otimes\widetilde{\bm{F}}_k&=\sum_{j=1}^{N_{\rm elem}}\sum_{k=0}^K
\widetilde{\bm{G}}_k\otimes\widetilde{\bm{F}}_k^{(j)} =\sum_{j=1}^{N_{\rm elem}}\sum_{k=0}^K
\widetilde{\bm{G}}_k\otimes \widetilde{a}_k^{(j)}\bm{F}^{(j)}\\
&=\sum_{j=1}^{N_{\rm elem}}\left(\sum_{k=0}^K
\widetilde{a}_k^{(j)}\widetilde{\bm{G}}_k\right)\otimes\bm{F}^{(j)} =: \sum_{j=1}^{N_{\rm elem}}\bm{M}^{(j)},
\end{align}
where $\widetilde{a}_k^{(j)}$ and $\widetilde{\bm{G}}_k\in\mathbb{R}^{N_{\rm P}\times N_{\rm P}}$ are such that the matrices $\bm{M}^{(j)}$ are positive semidefinite for all $j = 1,\ldots, N_\text{elem}$\rev{, and the resulting matrix $\bm{M}$ is positive-definite}. 

\rev{Note that while formally the structure of $\bm{M}$ is the same as that of the original matrix,  special choices of $\widetilde{\bm{G}}_k$, $k=0,\ldots,K$, can simplify the solves with $\bm{M}$ greatly, in comparison with the solves with $\bm{A}$.} We will see \rev{in \cref{sec:precond}} 
that many of the preconditioners that are used in practice are indeed of the form \cref{eq:matrixB}. \rev{Recall that since the preconditioner only differs from $\bm{A}$ in the stochastic part, we have to have an efficient solver for the underlying deterministic problem.}

The following theorem shows that the spectral equivalence between $\bm{A}$ and $\bm{M}$ can be obtained from the spectral equivalence between $\sum_{k=0}^Ka_k^{(j)}\bm{G}_k$ and $\sum_{k=0}^K\widetilde{a}_k^{(j)}\widetilde{\bm{G}}_k$ on each element
$\tau_j$, $j=1,\dots,N_{\rm elem}$. The obtained spectral bounds do not depend on the type and the number of the FE basis functions.

\begin{theorem}\label{th:1}
	Let the matrices $\bm{A}$ and $\bm{M}$ be defined by~\cref{eq:A_final} and~\cref{eq:matrixB}, respectively, and let $0<\underline{c}\le\overline{c}$  be such that 	\begin{equation}\label{eq:equiv_G}
	\underline{c}\,\bm{v}^T\left(\sum_{k=0}^K\widetilde{a}_k^{(j)}
	\widetilde{\bm{G}}_k\right)\bm{v}\le
	\bm{v}^T\left(\sum_{k=0}^Ka_k^{(j)}\bm{G}_k\right)\bm{v}\le \overline{c}\,
	\bm{v}^T\left(\sum_{k=0}^K\widetilde{a}_k^{(j)}
	\widetilde{\bm{G}}_k\right)\bm{v},\quad \text{for all }\;\bm{v}\in\mathbb{R}^{N_{\rm P}},
	\end{equation}
	$j =1,\ldots,N_{\rm elem}$. Then also
	\begin{equation}\label{eq:const_beta}
	\underline{c}\,\bm{v}^T\bm{Mv}\le \bm{v}^T\bm{Av}\le \overline{c}\,\bm{v}^T\bm{Mv},\quad \text{for all}\ \bm{v}\in\mathbb{R}^{N_{\rm FE}N_{\rm P}}.
	\end{equation}
\end{theorem}

The proof of \cref{th:1} is based on the two following lemmas.

\begin{lemma}\label{th:lem1}
	Let  $(\cdot,\cdot)_A$ and $(\cdot,\cdot)_M$ be two inner products on a Hilbert space $V$. Let the inner products be composed as
	\begin{equation}\label{eq:local_inprod_assumption}
	(u,v)_A=\sum_{j=1}^N(u,v)_{A,j},\quad (u,v)_M=\sum_{j=1}^N(u,v)_{M,j}, 
	\quad u,v\in V,
	\end{equation}
	where $(\cdot,\cdot)_{A,j}$ and $(\cdot,\cdot)_{M,j}$, $j=1,\ldots,N$, are positive semidefinite bilinear forms on $V$. Let there exist two positive real constants $\underline{c}$ and $\overline{c}$ such that the induced seminorms are uniformly equivalent in the following sense
	\begin{equation}\label{eq:ineq_elem}
	\underline{c}\,(u,u)_{M,j}\le (u,u)_{A,j}\le \overline{c}\,(u,u)_{M,j},\quad
	\text{for all}\; u\in V, \quad j = 1,\ldots,N.
	\end{equation}
	Then the induced (cumulative) norms are also equivalent with the same constants, i.e.,
	\begin{equation}\label{eq:ineq1}
	0<\underline{c}\le\frac{(u,u)_{A}}{(u,u)_{M}}\le\overline{c},\quad\text{for all}\; u\in V,\, u\neq0.
	\end{equation}
\end{lemma}
\begin{proof}
	The proof follows trivially from:
    \begin{align}
	\underline{c}\,(u,u)_{M}\overset{\cref{eq:local_inprod_assumption}}{=}\underline{c}\,\sum_{j=j}^N(u,u)_{M,j}&\overset{\cref{eq:ineq_elem}}{\le}\ \overbrace{\sum_{j=1}^N(u,u)_{A,j}}^{\overset{\cref{eq:local_inprod_assumption}}{=}(u,u)_{A}}\ \overset{\cref{eq:ineq_elem}}{\le} \overline{c}\,\sum_{j=1}^N(u,u)_{M,j}\overset{\cref{eq:local_inprod_assumption}}{=}\overline{c}\,(u,u)_{M}.
	\end{align}
\end{proof}

\begin{remark}
	\cref{th:lem1} can also be formulated in terms of matrices:
	If $\bm{A}=\sum_j\bm{A}^{(j)}$, $\bm{M}=\sum_j\bm{M}^{(j)}$ and
	$\underline{c}\,\bm{v}^T\bm{M}^{(j)}\bm{v}\le\bm{v}^T\bm{A}^{(j)}\bm{v}
	\le\overline{c}\,\bm{v}^T\bm{M}^{(j)}\bm{v}$ for all $\bm{v}$ and $j$,
	then  $\underline{c}\,\bm{v}^T\bm{M}\bm{v}\le\bm{v}^T\bm{A}\bm{v}
	\le\overline{c}\,\bm{v}^T\bm{M}\bm{v}$ for all $\bm{v}$.
\end{remark}

\begin{lemma}\label{th:lem2}
	Let $\bm{X},\bm{Y}\in \mathbb{R}^{N_1\times N_1}$ be symmetric positive definite and $\bm{Z}\in\mathbb{R}^{N_2\times N_2}$ be symmetric positive semidefinite.
	Let  
	\begin{equation}\underline{c}\,\bm{v}^T\bm{Y}\bm{v}\le \bm{v}^T\bm{X}\bm{v}\le \overline{c}\,\bm{v}^T\bm{Y}\bm{v},\quad
	\text{for all}\; \bm{v}\in \mathbb{R}^{N_1},
	\end{equation}
	hold for some positive real constants $\underline{c}$ and $\overline{c}$. Then also
	\begin{equation}\underline{c}\,\bm{u}^T(\bm{Y}\otimes\bm{Z})\bm{u}\le \bm{u}^T(\bm{X}\otimes\bm{Z})\bm{u}\le \overline{c}\,\bm{u}^T(\bm{Y}\otimes\bm{Z})\bm{u},\quad
	\text{for all}\; \bm{u}\in \mathbb{R}^{N_1N_2}.
	\end{equation}
\end{lemma}
\begin{proof}
	If $\bm{Z}$ is invertible, then the proof follows trivially from
	\begin{equation}\label{eq:kron_invert}
	(\bm{Y}\otimes \bm{Z})^{-1}(\bm{X}\otimes \bm{Z}) = (\bm{Y}^{-1}\otimes \bm{Z}^{-1})(\bm{X}\otimes \bm{Z}) = (\bm{Y}^{-1}\bm{X})\otimes (\bm{Z}^{-1}\bm{Z}) = (\bm{Y}^{-1}\bm{X})\otimes I,
	\end{equation}
	see, e.g., \cite[Section~13.3]{SIAM_Laub} or \cite[Section~4.1]{Powell2010Preconditioning}, and the fact that the spectra satisfy $\sigma((\bm{Y}^{-1}\bm{X})\otimes I)=\sigma(\bm{Y}^{-1}\bm{X})$.
	If $\bm{Z}$ is singular, then $\bm{X}\otimes \bm{Z}$ as well as $\bm{Y}\otimes \bm{Z}$ are invertible on $\mathcal{R}(I\otimes(\bm{Z}^\dagger \bm{Z}))$ and zero on $\mathcal{R}(I\otimes(I - \bm{Z}^\dagger \bm{Z}))$, where $^\dagger$ denotes the Moore-Penrose pseudoinverse. Combined with \cref{eq:kron_invert}, the proof follows directly.
\end{proof}

We are now ready to prove \cref{th:1}.

\begin{proof}[Proof of \cref{th:1}]
	Under the assumption \cref{eq:omega_max} and the assumption on the preconditioning matrix, the matrices $\sum_{k=0}^Ka_k^{(j)}\bm{G}_k$ and $\sum_{k=0}^K\widetilde{a}_k^{(j)}\widetilde{\bm{G}}_k$, $j = 1,\ldots, N_\text{elem}$, are positive definite, while $\bm{F}^{(j)}$, $j = 1,\ldots, N_\text{elem}$, are positive semidefinite. From \cref{eq:equiv_G} and \cref{th:lem2}, we obtain uniform spectral equivalence between $\bm{A}^{(j)}$ and $\bm{M}^{(j)}$ on each element $\tau_j$. Applying \cref{th:lem1} to the seminorms defined by $\bm{A}^{(j)}$ and $\bm{M}^{(j)}$ finishes the proof.
\end{proof}

Let us demonstrate on an example how $\underline{c}$ and $\overline{c}$ are obtained using \cref{th:1}.
\begin{example}\label{ex:4}
	Let $\bm{A}$ be the matrix from \cref{ex:1}
	and let the preconditioner $\bm{M}$ be block-diagonal, i.e.,
	\begin{equation}
	\bm{M}= \sum_{k=0}^K\widetilde{\bm{G}}_k\otimes\bm{F}_k := \bm{G}_0\otimes \bm{F}_0.
	\end{equation}
	For the element $\tau_{j}$, $j=1,\dots,N_{\rm elem}$, we have 
	\begin{equation}\label{eq:ex2}
	\sum_{k=0}^Ka_k^{(j)}\bm{G}_k=\left(\begin{array}{ccc}
	a_0^{(j)} &\frac{1}{\sqrt{3}}a_1^{(j)} & 0\\
	\frac{1}{\sqrt{3}}a_1^{(j)}& a_0^{(j)}&\frac{2}{\sqrt{15}}a_1^{(j)} \\
	0&\frac{2}{\sqrt{15}}a_1^{(j)} &a_0^{(j)} \\
	\end{array}\right)\quad \text{and} \quad
	\sum_{k=0}^K\widetilde{a}_k^{(j)}\widetilde{\bm{G}}_k = \left(\begin{array}{ccc}
	a_0^{(j)} &0 & 0\\
	0& a_0^{(j)}&0 \\
	0&0  &a_0^{(j)} \\
	\end{array}\right).
	\end{equation}
	For $\vert a_1^{(j)}\vert\le  a_0^{(j)}$, $j = 1,\ldots,N_\text{elem}$, it is easy to prove that these matrices 
	are spectrally equivalent with 
	\begin{equation}
	\underline{c}=1-\frac{\sqrt{15}}{5}\quad\text{and}\quad
	\overline{c}= 1+\frac{\sqrt{15}}{5}.
	\end{equation}
	Thus, using \cref{th:1}, also
	\begin{equation}
	\left(1-\frac{\sqrt{15}}{5}\right)\bm{v}^T\bm{M}\bm{v}\le \bm{v}^T\bm{A}\bm{v}\le \left(1+\frac{\sqrt{15}}{5}\right)\bm{v}^T\bm{M}\bm{v}, \quad \text{for all}\ \bm{v}\in\mathbb{R}^{N_{\rm FE}N_{\rm P}},
	\end{equation}
	and
	\begin{equation}\label{eq:ex_cond}
	\kappa(\bm{M}^{-1}\bm{A})\le 4+\sqrt{15}\approx 7.87.
	\end{equation}
	In other words, if $\vert a_1(\bm{x})\vert\le a_0(\bm{x})$, i.e.,~$\mu=1$ in~\cref{eq:omega}, the block-diagonal preconditioning of $\bm A$ from \cref{eq:ex1} yields the condition number \cref{eq:ex_cond}. If $\vert a_1(\bm{x})\vert\le \frac{1}{2}a_0(\bm{x})$, i.e.,~$\mu=\frac{1}{2}$ in~\cref{eq:omega},
	we analogously get 
	\begin{equation}\label{eq:226}
	\kappa(\bm{M}^{-1}\bm{A})\le \frac{23+4\sqrt{15}}{17}
	\approx 2.26.
	\end{equation}
	Using the classical assumption~\cref{eq:stronger_assum} and not employing the information about the spectrum of $\bm{G}_1$, we obtain
	for $\vert a_1(\bm{x})\vert\le\frac{1}{2} a_0(\bm{x})$ the estimate $\kappa(\bm{M}^{-1}\bm{A})\le 3$, while for $\vert a_1(\bm{x})\vert\le a_0(\bm{x})$, the term $\kappa(\bm{M}^{-1}\bm{A})$ cannot be \rev{bounded. In}~\cite{Powell2009Block}, one of the pioneering papers on spectral estimates of preconditioned stochastic Galerkin matrices, the spectral bounds for specific 
	$s$ and $K$ were derived, which for $K=1$,
	${\mu}=1$ and ${\mu}=\frac{1}{2}$ result in~\cref{eq:ex_cond} and~\cref{eq:226}, respectively. 
	These bounds, however, do not reflect
	possible local properties of functions $a_k$, and thus for $K\ge 2$,
	they are still less accurate than our estimates, as will be shown in \cref{sec:num_exp}.
\end{example}

Note that the approach for obtaining spectral equivalence of the stochastic Galerkin matrices presented in this section and summarized in \cref{th:1} is independent of the choice of the approximation spaces. In the following section, we apply these results to approximation spaces introduced in \cref{sec:approx_spaces} and some inner products of the form \cref{eq:matrixB} defined on them to obtain new spectral bounds of the related preconditioned system matrix $\bm{M}^{-1}\bm{A}$.

\section{Preconditioning and spectral bounds}\label{sec:precond}

In the following part, we present three inner products
$(\cdot,\cdot)_M$ on $V$ and their matrix representations $\bm{M}$
that can serve as preconditioners for $\bm{A}$. For each of them, we compute constants $\underline{c}$ and $\overline{c}$ defined in~\cref{eq:const_beta}, which bound the spectrum of the preconditioned matrices $\bm{M}^{-1}\bm{A}$. 

The inner products and the corresponding matrices $\bm{M}$ presented in this section should only serve as examples. Other preconditioning matrices of the form \cref{eq:matrixB} can be studied and corresponding bounds to the resulting spectra can be derived analogously.

\subsection{Mean-based preconditioning}\label{sec:mean-based}
Due to~\cref{eq:omega}, we can approximate the original inner product by an inner product, where the function $a(\bm{x},\bm{\xi})$ is substituted by $a_0(\bm{x})$, representing for centralized distributions of $\bm{\xi}$ the mean value of $a(\bm{x},\bm{\xi})$. The \emph{mean-based inner product} is defined as
\begin{equation}
(u,v)_{M}=\int_{{\mathbb R}^K}\int_D a_0(\bm{x})\nabla u(\bm{x},\bm{\xi})\cdot
\nabla v(\bm{x},\bm{\xi})\rho(\bm{\xi})
\,{\rm d}\bm{x}\,{\rm d}\bm{\xi}.
\end{equation}
The corresponding preconditioning matrix
\begin{equation}\label{eq:mean_based_B}
\bm{M}=\bm{G}_0\otimes\bm{F}_0
\end{equation}
is then block-diagonal with the diagonal blocks of size $N_{\rm FE}\times N_{\rm FE}$.

This type of preconditioning can be used both for the complete and the tensor product polynomials;
see, e.g., \cite{Powell2009Block,PowellSilvesterSim,Rosseel,Ullmann2010Kronecker}. 
We first derive the bounds $\underline{c}$ and $\overline{c}$ for 
the tensor product polynomials. These bounds also apply to the complete polynomials because 
$V^{\rm C}_{s}\subset V^{\rm TP}_{s,\dots,s}$.

\begin{example} 
	Consider the setting from \cref{ex:2} and \cref{ex:3}. The non-zero patterns of the preconditioning matrices defined in \cref{eq:mean_based_B} are: 
	$$\bm{M}^{\rm TP}={\footnotesize \left(\begin{array}{ccc|ccc|ccc}
	X&& &&&&&&\\
	&X&&&&&&&\\
	&&X&&&&&&\\
	\hline
	&&&X&& & & & \\
	&&&&X&& & & \\
	&&& &&X& & & \\
	\hline
	&&&&&&X&& \\
	&&&&&&&X&\\
	&&&&&&&&X\\
	\end{array}\right)}\,, \quad 
	\bm{M}^C={\footnotesize \left(\begin{array}{c|cc|ccc}
	X&&&&&\\
	\hline
	&X&&&&\\
	&&X&&&\\
	\hline
	&&&X&&\\
	&&&&X&\\
	&&&&&X
	\end{array}\right)}\,,$$
	where $X$ stands for a non-zero block of size $N_{\rm FE}\times N_{\rm FE}$.
\end{example}

\begin{lemma}\label{th:truediag}
	Under assumption \cref{eq:omega} and for $\bm M$ defined by~\cref{eq:mean_based_B},
	the constants $\underline{c}$ and $\overline{c}$ 
	defined in~\cref{eq:const_beta} can be obtained as
	\begin{equation}\label{eq:lambda_min_max}
	\underline{c}=1-\mu\,\lambda_{\text{max}}(\bm{G}_{s,1}) \quad \text{and} \quad \overline{c}=1+\mu\,\lambda_{\text{max}}(\bm{G}_{s,1}),
	\end{equation}
	where $s=\max(s_1,\dots,s_K)$, and $\lambda_{\text{max}}(\bm{G}_{s,1})$ is the largest eigenvalue of $\bm{G}_{s,1}$ defined in~\cref{eq:G_small}.
\end{lemma}
\begin{proof}
	The proof consists of several rather straightforward steps, which we however prefer to give in full detail, since analogous technique will be used in the proofs of some of the subsequent lemmas.
	
	Using \cref{th:1}, we only need to prove that for every $\bm{v}\in \mathbb{R}^{N_{\rm P}}$, it holds that
	\begin{equation}\label{eq:to_be_shown}
	\underline{c}\,a_0^{(j)}\bm{v}^T\bm{G}_0\bm{v}\le
	\bm{v}^T\left(a_0^{(j)}\bm{G}_0+\sum_{k=1}^Ka_k^{(j)}\bm{G}_k\right)\bm{v}
	\le\overline{c}\,a_0^{(j)}\bm{v}^T\bm{G}_0\bm{v}, \quad j = 1,\ldots,N_\text{elem}.
	\end{equation}
	Since $\bm{G}_{s,0}=\bm{I}_{s}$, equations \cref{eq:lambda_min_max}  imply
	\begin{equation}
	\underline{c}\,\bm{v}^T\bm{G}_{s,0}\bm{v}\le
	\bm{v}^T\bm{G}_{s,0}\bm{v}\pm\mu\,\bm{v}^T\bm{G}_{s,1}\bm{v}\le
	\overline{c}\,\bm{v}^T\bm{G}_{s,0}\bm{v}, \quad \bm{v}\in \mathbb{R}^{s}.
	\end{equation}
	Due to the interlacing property of the eigenvalues of Jacobi matrices, we immediately obtain
	\begin{equation}
	\underline{c}\,\bm{v}^T\bm{G}_{s_k,0}\bm{v}\le
	\bm{v}^T\bm{G}_{s_k,0}\bm{v}\pm\mu\,\bm{v}^T\bm{G}_{s_k,1}\bm{v}\le
	\overline{c} \bm{v}^T\bm{G}_{s_k,0}\bm{v}, \quad \bm{v}\in \mathbb{R}^{s_k},
	\end{equation}
	for all  $s_k\leq s$. Using the tensor structure \cref{eq:AsumGGF} of the matrices $\bm{G}_k$, we get
	\begin{equation}
	\underline{c}\,\bm{v}^T\bm{G}_0\bm{v}\le
	\bm{v}^T\bm{G}_0\bm{v}\pm\mu\,\bm{v}^T\bm{G}_k\bm{v}\le
	\overline{c}\, \bm{v}^T\bm{G}_0\bm{v}, \quad k = 1,\dots,K, \quad \bm{v}\in \mathbb{R}^{N_{\rm P}}.
	\end{equation}
	Multiplying by $|a^{(j)}_k|$, we obtain
	\begin{equation}
	\underline{c}\,\bm{v}^T|a_k^{(j)}|\bm{G}_0\bm{v}\le\bm{v}^T|a_k^{(j)}|\bm{G}_0\bm{v}+\mu\,a_k^{(j)}\bm{v}^T\bm{G}_k\bm{v}\le\overline{c}\,\bm{v}^T|a_k^{(j)}|\bm{G}_0\bm{v}, \quad k = 1,\dots,K,
	\end{equation}
	which taking sum over $k$ becomes
	\begin{equation}\label{eq:first_component_I}
	\underline{c}\,\bm{v}^T\left(\sum_{k=1}^K|a_k^{(j)}|\bm{G}_0\right)\bm{v}\le
	\bm{v}^T\left(\sum_{k=1}^K|a_k^{(j)}|\bm{G}_0+\mu \sum_{k=1}^Ka_k^{(j)}\bm{G}_k\right)\bm{v}\le
	\overline{c}\,\bm{v}^T\left(\sum_{k=1}^K|a_k^{(j)}|\bm{G}_0\right)\bm{v}.
	\end{equation}
	Due to \cref{eq:omega} and the fact that $\underline{c}\leq 1\leq\overline{c}$, it also holds that
	\begin{align}\label{eq:third_component_I}
	\underline{c}\,\bm{v}^T\left(\mu\, a_0^{(j)} - \sum_{k=1}^K\vert a_k^{(j)}\vert\right)\bm{G}_0\bm{v}\leq \bm{v}^T&\left(\mu\, a_0^{(j)} - \sum_{k=1}^K\vert a_k^{(j)}\vert\right)\bm{G}_0\bm{v}\\
	& \hspace{2cm} \leq\overline{c}\,\bm{v}^T\left(\mu\, a_0^{(j)} - \sum_{k=1}^K\vert a_k^{(j)}\vert\right)\bm{G}_0\bm{v}.
	\end{align}
	Adding \cref{eq:first_component_I} and \cref{eq:third_component_I}, we get
	\begin{equation}
	\underline{c}\, \mu\,a_0^{(j)}\bm{v}^T\bm{G}_0\bm{v}\le
	\bm{v}^T\left(\mu\,a_0^{(j)}\bm{G}_0+
	\mu\sum_{k=1}^K a_k^{(j)}\bm{G}_k\right)\bm{v}\le
	\overline{c}\, \mu\,a_0^{(j)}\bm{v}^T\bm{G}_0\bm{v}.
	\end{equation}
	By dividing by $\mu>0$, we obtain the desired inequality~\cref{eq:to_be_shown}. 
\end{proof}

Since the eigenvalues of the Jacobi matrix $\bm{G}_{s,1}$ are the roots of the polynomial $\psi_{s}$, the spectral bounds $\underline{c}$ and $\overline{c}$ can be obtained directly from the maximal roots of the polynomial $\psi_{s}$, which we denote by $\lambda_{\text{max}}(\psi_{s})$. \rev{Thanks to this relation, we can formulate the following corollary purely in terms of these extremal roots.}
\begin{corollary}\label{cor1}
	Let \cref{eq:omega} be satisfied and let the matrix $\bm{M}$ represent the mean-based preconditioning \cref{eq:mean_based_B}. Then
	\begin{equation}\label{eq:cor1}
	\kappa(\bm{M}^{-1}\bm{A}) \leq \frac{1+\mu\,\lambda_{\text{max}}(\psi_{s})}{1 -\mu\,\lambda_{\text{max}}(\psi_{s})},
	\end{equation} 
	where for tensor product polynomials, we define $s=\max(s_1,\dots,s_K)$. 
\end{corollary}
\rev{Note that if $a$ is not significantly dominated by the term $a_0$ in the sense of \cref{eq:omega}, we can expect the mean-based preconditioning to perform rather poorly, see also \cite{Powell2009Block}. This is reflected in the bound \cref{eq:cor1} by the denominator $1 -\mu\,\lambda_{\text{max}}(\psi_{s})$ being close to zero.}
\rev{
\begin{remark}
It is interesting to note that the obtained equivalence constants $0<\underline{c}\le\overline{c}$ can be also used for a~posteriori estimates of the energy norm $\bm{e}^T\bm{Ae}$ of the algebraic error. Let $\bm{e}:=\bm{x}-\widetilde{\bm{x}}$ be the error of inexact solutions $\widetilde{\bm{x}}$ of the linear systems $\bm{Ax}=\bm{b}$ and let
\begin{equation}\underline{c}\,\bm{v}^T\bm{Mv}\le \bm{v}^T\bm{Av}\le \overline{c}\,\bm{v}^T\bm{Mv}
\end{equation} 
hold for all $\bm{v}$ of appropriate size. If solutions of the system with $\bm{M}$ are easily accessible, then due to $\bm{e}^T\bm{Ae}=\bm{r}^T\bm{A}^{-1}\bm{r}$, $\bm{r}:=\bm{Ae}$, the bounds to the error can be obtained efficiently as
\[\frac{1}{\overline{c}}\,\bm{r}^T\bm{M}^{-1}\bm{r}\le \bm{r}^T\bm{A}^{-1}\bm{r}\le \frac{1}{\underline{c}}\,\bm{r}^T\bm{M}^{-1}\bm{r};
\]
see also \cite[Theorem~5.3]{AinsworthOden}.
\end{remark}
}

\subsection{Preconditioning using truncated expansion of \texorpdfstring{{\boldmath$a$}}{a}}

Instead of the block-diagonal matrix \cref{eq:mean_based_B}, we can consider a block-diagonal matrix with larger blocks. \rev{This strategy can be advantageous especially if only a small number of parallel processes can be employed.}

If we consider the tensor product polynomials, we can obtain a new inner product by omitting the last term of the expansion \cref{eq:a_expansion} of $a(\bm{x},\bm{\xi})$, i.e.,
\begin{equation}
(u,v)_M=\int_{{\mathbb R}^K}\int_D
\left(a_0(\bm{x})+\sum_{k=1}^{K-1} a_k(\bm{x})\xi_k\right)
\nabla u(\bm{x},\bm{\xi})\cdot
\nabla v(\bm{x},\bm{\xi})\rho(\bm{\xi})
\,{\rm d}\bm{x}\,{\rm d}\bm{\xi}.
\end{equation}
The corresponding preconditioning matrix
\begin{equation}\label{eq:K-1_term_B}
\bm{M}=\sum_{k=0}^{K-1}\bm{G}_k\otimes \bm{F}_k
\end{equation}
is then block-diagonal with the diagonal blocks of size $N_{\rm FE}\cdot \prod_{k=1}^{K-1}s_k
\times N_{\rm FE}\cdot \prod_{k=1}^{K-1}s_k$.

\begin{example} 
	Consider the setting from \cref{ex:2}. The non-zero pattern of the preconditioning matrix defined in \cref{eq:K-1_term_B} is:
	\begin{equation}
	\bm{M}^{\rm TP}={\footnotesize \left(\begin{array}{ccc|ccc|ccc}
	X&X& &&&&&&\\
	X&X&X&&&&&&\\
	&X&X&&&&&&\\
	\hline
	&&&X&X& & & & \\
	&&&X&X&X& & & \\
	&&& &X&X& & & \\
	\hline
	&&&&&&X&X& \\
	&&&&&&X&X&X\\
	&&&&&&&X&X\\
	\end{array}\right)}.
	\end{equation}
\end{example}

\rev{Note that one can consider many other truncation schemes, when a various number of terms is truncated from the expansion \cref{eq:a_expansion}. The mean-based preconditioning and the preconditioning \cref{eq:K-1_term_B} represent only two extreme cases of this strategy. After proper reordering, either of the expansion \cref{eq:a_expansion} or the resulting matrix $\bm{A}$, we will again get a block diagonal preconditioner $\bm{M}$. The efficiency of omitting a particular term depends on a specific setting (the corresponding stochastic approximation space).} 

\rev{It is also possible to apply this technique to complete polynomials. However, if we use the natural ordering 
of complete polynomials, the preconditioning matrix will not have a block diagonal form. If we consider complete polynomials ordered as in \cref{eq:ex_TP_C_A}, the resulting bound will be the same as in the case of tensor-product polynomials.}

\begin{lemma}\label{th:truncated}
	Under assumption \cref{eq:omega} and for $\bm M$ defined by~\cref{eq:K-1_term_B}, the constants $\underline{c}$ and $\overline{c}$ defined in~\cref{eq:const_beta} can be obtained as
	\begin{equation}\label{eq:lambda_min_max_2}
	\underline{c}=1-\mu\,\lambda_{\text{max}}(\bm{G}_{s_K,1}) \quad \text{and} \quad \overline{c}=1+\mu\,\lambda_{\text{max}}(\bm{G}_{s_K,1}),
	\end{equation}
	where $\lambda_{\text{max}}(\bm{G}_{s_K,1})$ is the largest eigenvalue of 
	$\bm{G}_{s_K,1}$ defined in~\cref{eq:G_small}. 
\end{lemma}
\begin{proof}
	Using \cref{th:1}, we only need to prove for every $\bm{v}\in \mathbb{R}^{N_{\rm P}}$ 
	\begin{equation}\label{eq:lemma4_c}
	\underline{c}\,\bm{v}^T\left(\sum_{k=0}^{K-1}{a}_k^{(j)}{\bm{G}}_k\right)\bm{v}\le\bm{v}^T\left(\sum_{k=0}^K a_k^{(j)}\bm{G}_k\right)\bm{v}\le \overline{c}\,\bm{v}^T\left(\sum_{k=0}^{K-1}{a}_k^{(j)}{\bm{G}}_k\right)\bm{v},\quad j=1,\dots,N_{\rm elem}.
	\end{equation}
	Analogously to \cref{th:truediag}, equations \cref{eq:lambda_min_max_2} imply
	\begin{equation}\label{eq:first_component_II}
	\underline{c}\,\bm{v}^T|a_K^{(j)}|\bm{G}_0\bm{v}\le
	\bm{v}^T|a_K^{(j)}|\bm{G}_0\bm{v}+\mu\, a_K^{(j)}\bm{v}^T\bm{G}_K\bm{v}\le
	\overline{c}\,\bm{v}^T|a_K^{(j)}|\bm{G}_0\bm{v}.
	\end{equation}
	From the definition of $\mu$, we have 
	\begin{equation}\label{eq:first_for_splitting}
	\bm{v}^T\left(\bm{G}_{0}\pm\mu\,\bm{G}_{k}\right)\bm{v}\ge 0, \quad k = 1,\dots,K,
	\end{equation}
	which multiplying by $|a_k^{(j)}|$ and  taking sum over $k = 1,\ldots,K-1$ becomes
	\begin{equation}\label{eq:aux_22}
	\bm{v}^T\left(\sum_{k=1}^{K-1}\vert a_k^{(j)}\vert\bm{G}_{0}+\mu 
	\sum_{k=1}^{K-1}a_k^{(j)} \bm{G}_{k}\right)\bm{v}\ge 0.
	\end{equation}
	Using $\underline{c}\leq 1\leq\overline{c}$, we get from \cref{eq:aux_22}
	\begin{align}\label{eq:second_component_II}
	\underline{c}\,\bm{v}^T\left(
	\sum_{k=1}^{K-1}\vert a_k^{(j)}\vert \bm{G}_0+
	\mu\sum_{k=1}^{K-1}{a}_k^{(j)}
	{\bm{G}}_k\right)\bm{v}\le &
	\, \bm{v}^T\left(
	\sum_{k=1}^{K-1}\vert a_k^{(j)}\vert \bm{G}_0+
	\mu\sum_{k=1}^{K-1}{a}_k^{(j)}
	{\bm{G}}_k\right)\bm{v}\\
	&\le \overline{c}\,
	\bm{v}^T\left(
	\sum_{k=1}^{K-1}\vert a_k^{(j)}\vert \bm{G}_0+
	\mu\sum_{k=1}^{K-1}{a}_k^{(j)}
	{\bm{G}}_k\right)\bm{v}.
	\end{align}
	Adding \cref{eq:first_component_II}, \cref{eq:second_component_II}, and \cref{eq:third_component_I}, we finally obtain 
	\begin{align}\label{eq:last_for_splitting}
	\underline{c}\,\bm{v}^T\left(
	\mu\,a_0^{(j)} \bm{G}_0+
	\mu\sum_{k=1}^{K-1}{a}_k^{(j)}
	{\bm{G}}_k\right)\bm{v}\le
	\bm{v}^T&\left(
	\mu\,a_0^{(j)} \bm{G}_0+
	\mu\sum_{k=1}^{K}{a}_k^{(j)}
	{\bm{G}}_k\right)\bm{v}\\ &\hspace{.5cm}\le \overline{c}\,
	\bm{v}^T\left(
	\mu\, a_0^{(j)} \bm{G}_0+
	\mu\,\sum_{k=1}^{K-1}{a}_k^{(j)}
	{\bm{G}}_k\right)\bm{v},
	\end{align}
	which dividing by $\mu$ yields the desired inequality~\cref{eq:lemma4_c}.
\end{proof}

Using this lemma, the spectral bounds $\underline{c}$ and $\overline{c}$ can be obtained directly from the roots of the polynomial $\psi_{s_K}$, similarly as in the previous section.
\begin{corollary}
	Let \cref{eq:omega} be satisfied and let the matrix $\bm{M}$ represent the $(K-1)$-term expansion preconditioning \cref{eq:K-1_term_B}. Then
	\begin{equation}
	\kappa(\bm{M}^{-1}\bm{A}) \leq \frac{1+\mu\,\lambda_{\text{max}}(\psi_{s_K})}{1 -\mu\,\lambda_{\text{max}}(\psi_{s_K})}.
	\end{equation}
\end{corollary}

\subsection{Splitting-based preconditioning}\label{sec:splitting_based}

Another inner product can be obtained by splitting the approximation space $V$ into two complementary subspaces, $V=U\oplus W$, $U\cap W=0$, so that any $v\in V$ can be uniquely decomposed as $v=v_U+v_W$, $v_U\in U$, $v_W\in W$. Using this decomposition, we can define the new inner product component-wise, i.e., 
\begin{equation}\label{eq:splitting_based}
({u},{v})_M=({u}_U,{v}_U)_A+({u}_W,{v}_W)_A.
\end{equation}
For the space $V^{\rm TP}_{s_1,\dots,s_K}$ of the tensor product polynomials of the degrees $s_1-1,\dots,s_K-1$, we can use the splitting $U=V^{\rm TP}_{s_1,\dots,s_{K-1},s_K-1}$ and $W$ such that $V^{\rm TP}_{s_1,\dots,s_K}=U\oplus W$, i.e.,~$W$ contains the polynomials of $V^{\rm TP}_{s_1,\dots,s_K}$ of degree exactly $s_K-1$ in the variable $\xi_K$. The corresponding preconditioning matrix $M$ has then a two-by-two block-diagonal 
form, see, e.g.,~\cite{Rosseel,Sousedik2014Hierarchical},
and can be obtained as 
\begin{equation}\label{eq:MinLemma5}
\bm{M}=\sum_{k=0}^{K-1}
{\bm{G}}_k\otimes \bm{F}_k+\widetilde{\bm{G}}_K\otimes \bm{F}_K, \quad 
\widetilde{\bm{G}}_K=\widetilde{\bm{G}}_{s_K,1}\otimes \bm{G}_{s_{K-1},0}\otimes \dots \otimes \bm{G}_{2,0}\otimes \bm{G}_{1,0},
\end{equation}
where the matrix $\widetilde{\bm{G}}_{s_K,1}$ is obtained from ${\bm{G}}_{s_K,1}$ by annihilating the very last elements in both the sub- and super-diagonal. For distributions with $\alpha_n = 0$, $n=1,2,\ldots$ in the recurrence \cref{eq:const_bj_cj}, this matrix satisfies
\begin{equation}
\widetilde{\bm{G}}_{s_K,1} = \begin{pmatrix}\bm{G}_{s_K-1,1}&\\&0\end{pmatrix}.
\end{equation}

For the space $V^{\rm C}_{s}$ of the complete polynomials of the total degree 
at most $s-1$, we can use the splitting $U=V^{\rm C}_{s-1}$ and $W=W^{\rm C}_{s}$, where $W^{\rm C}_{s}$ is the span of the complete polynomials of the total degree exactly $s-1$. Similarly to the previous case, the corresponding preconditioning matrix $\bm M$ has then a two-by-two block-diagonal form and can be obtained as
\begin{equation}\label{eq:MinLemma6}
\bm{M}=\sum_{k=0}^{K}{\widetilde{\bm{G}}}_k\otimes \bm{F}_k=
\bm{G}_0\bm{F}_0+\sum_{k=1}^{K}{\widetilde{\bm{G}}}_k\otimes \bm{F}_k,
\end{equation}
where the matrices $\widetilde{\bm{G}}_k$ coincide with $\bm{G}_k$ up to the last sub- and super-diagonal blocks of the sizes $(N^{\rm C}_{s-1}-N^{\rm C}_{s-2})\times(N_{s}^{\rm C}-N^{\rm C}_{s-1})$ which are annihilated in $\widetilde{\bm{G}}_k$.

If $U$ and $W$ are close to orthogonal, the preconditioning based on the splitting \cref{eq:splitting_based} enables to estimate the error reduction when the approximation space $U$ is enriched by $W$ in the Galerkin method. This can be exploited in adaptive algorithms, where $W$ is sometimes called the `detail' space; \rev{see Remark~\ref{rem413} and,} e.g.,~\cite{Bespalov2014Energy,Crowder2018Efficient,Pultarova2015Adaptive}.

\begin{example}
	Consider the setting from \cref{ex:2} and \cref{ex:3}. The non-zero patterns of the preconditioning matrices defined in \cref{eq:MinLemma5} and \cref{eq:MinLemma6} are:
	\begin{equation}
	\bm{M}^{\rm TP}={\footnotesize \left(\begin{array}{ccc|ccc|ccc}
	X&X& &X& & & & & \\
	X&X&X& &X& & & & \\
	&X&X& & &X& & & \\
	\hline
	X & & &X&X& & & & \\
	& X& &X&X&X& & & \\
	& & X& &X&X& & & \\
	\hline
	&&&&&&X&X& \\
	&&&&&&X&X&X\\
	&&&&&&&X&X\\
	\end{array}\right)}\,,\quad
	\bm{M}^{\rm C}={\footnotesize \left(\begin{array}{c|cc|ccc}
	X&X&X& & & \\
	\hline
	X&X& & & & \\
	X& &X& & & \\
	\hline
	& & &X& & \\
	& & & &X& \\
	& & & & &X
	\end{array}\right)}.
	\end{equation}
\end{example}

As we will see later in this section, the efficiency of the splitting-based preconditioner, both for the tensor-product and the complete polynomials, is determined by the spectral properties of the matrix
\begin{equation}\label{eq:mat_H}
\bm{H}^\pm_s=\left(\bm{I}_{s}\pm\mu\begin{pmatrix}\bm{G}_{s-1,1}&\\&0\end{pmatrix}\right)^{-1}\left(\bm{I}_{s}\pm\mu\,\bm{G}_{s,1}\right),
\end{equation}
where either $+$ or $-$ sign is considered.
We first investigate the spectrum of $\bm{H}^\pm_s$. In the subsequent lemma, we link the spectral properties of $\bm{M}^{-1}\bm{A}$ and that of $\bm{H}^\pm_s$.

\begin{lemma}\label{lemma49x}
	At most two eigenvalues of $\bm{H}^\pm_s$ defined in \cref{eq:mat_H} are different from 
	unity. These two eigenvalues do not depend on the sign considered and can be obtained as 
	\begin{equation}
	\lambda_{\min}=1-\sqrt{1-d_{s}}, \quad \lambda_{\max}=1+\sqrt{1-d_s},
	\end{equation}
	where
	\begin{equation}
	d_{s} = \frac{1}{e_{s}^T(\bm{I}_{s}+\mu\,\bm{G}_{s,1})^{-1}e_{s}}<1.
	\end{equation}
	Finally, for a fixed $s$, $d_s$ increases with $\mu\ge 0$ decreasing.
\end{lemma}
\begin{proof}
	Since 
	\begin{equation}\label{eq:rank_2_update}
	\left(\bm{I}_{s}\pm\mu\,\bm{G}_{s,1}\right) =\left(\bm{I}_{s}\pm\mu\begin{pmatrix}\bm{G}_{s-1,1}&\\&0\end{pmatrix}\right) \pm \mu\sqrt{\beta_{\rev{s-1}}}(e_{s-1}e_s^T + e_se_{s-1}^T),
	\end{equation}
	i.e., the two matrices differ by a rank-2 matrix, the first statement follows directly. Further, using \cref{eq:rank_2_update}, we have
	\begin{align}\label{eq:rank_2_update_inverse}
	&\left(\bm{I}_{s}\pm\mu\begin{pmatrix}\bm{G}_{s-1,1}&\\&0\end{pmatrix}\right)^{-1}\left(\bm{I}_{s}\pm\mu\,\bm{G}_{s,1}\right)\\ &\hspace{3cm}= \bm{I}_s \pm \mu\sqrt{\beta_{\rev{s-1}}}\left(\bm{I}_{s}\pm\mu\begin{pmatrix}\bm{G}_{s-1,1}&\\&0\end{pmatrix}\right)^{-1}(e_{s-1}e_{s}^T + e_se_{s-1}^T)\\
	&\hspace{3cm}=\bm{I}_s \pm \mu\,\sqrt{\beta_{\rev{s-1}}}\begin{pmatrix}(\bm{I}_{s-1}\pm\mu\bm{G}_{s-1,1})^{-1}&\\&1\end{pmatrix}(e_{s-1}e_{s}^T + e_se_{s-1}^T).
	\end{align}
	Therefore, to obtain those non-unit eigenvalues, it suffices to investigate the last two-by-two diagonal block of \cref{eq:rank_2_update_inverse}, which has the form
	\begin{equation}
	\begin{pmatrix}
	1 & \pm\mu\,\sqrt{\beta_{\rev{s-1}}}e_{s-1}^T(\bm{I}_{s-1}\pm\mu\bm{G}_{s-1,1})^{-1}e_{s-1}\\
	\pm\mu\sqrt{\beta_{\rev{s-1}}} & 1
	\end{pmatrix}.
	\end{equation}
	The eigenvalues therefore satisfy 
	\begin{align}
	\lambda_{\min}&=1-\sqrt{\mu^2\beta_{\rev{s-1}}e_{s-1}^T(\bm{I}_{s-1}\pm\mu\bm{G}_{s-1,1})^{-1}e_{s-1}},\\ \lambda_{\max}&=1+\sqrt{\mu^2\beta_{\rev{s-1}}e_{s-1}^T(\bm{I}_{s-1}\pm\mu\bm{G}_{s-1,1})^{-1}e_{s-1}}.
	\end{align}
	Defining $d_{j}^{-1} = e_{j}^T(\bm{I}_{j}+\mu\bm{G}_{j,1})^{-1}e_{j}$ and applying the recursive formula for the last element of a symmetric tridiagonal matrix, see \cite[Theorem~2.3]{Meurant92}, we get
	\begin{equation}\label{eq:recursive_form}
	d_1=1,\quad d_j=1-\frac{\mu^2\beta_{\rev{j-1}}}{d_{j-1}},\quad j=2,3,\ldots,
	\end{equation}
	and analogously for the $-$ sign, from which the second statement is obtained. The monotonicity of $d_s$ in $\mu$ follows from the discussion below, in particular \cref{eq:quadrature_symmetry}.
\end{proof}

\begin{remark}\label{rem:H_sign}
	Since the eigenvalues of $\bm{H}_{s}^{\pm}$ are independent of the sign, to simplify the notation, we further work with the matrix
	\begin{equation}
	\bm{H}_s := \bm{H}^+_s.
	\end{equation}
\end{remark}

\begin{lemma}\label{th:splitting_TP}
	Assume the tensor product polynomials. Under assumption \cref{eq:omega} and for $\bm{M}$ defined by \cref{eq:MinLemma5}, the constants $\underline{c}$ and $\overline{c}$ defined in~\cref{eq:const_beta} can be obtained as 
	\begin{equation}\label{eq:lemma_splitting}
	\underline{c}=\lambda_{\text{min}}(\bm{H}_{s_K}) \quad \text{and} \quad \overline{c}=\lambda_{\text{max}}(\bm{H}_{s_K}),
	\end{equation}
	where $\lambda_{\text{min}}(\bm{H}_{s_K})$ and $\lambda_{\text{max}}(\bm{H}_{s_K})$ are the smallest and the largest eigenvalue, respectively, of $\bm{H}_{s_K}$ defined in \cref{eq:mat_H}.
\end{lemma}
\begin{proof}
	Using \cref{th:1}, we only need to prove for every $\bm{v}\in \mathbb{R}^{N_{\rm P}}$, $j=1,\dots,N_{\rm elem}$,
	\begin{equation}\label{eq:G_for_splitting}
	\underline{c}\,\bm{v}^T\left(\sum_{k=0}^{K-1}{a}_k^{(j)}
	{\bm{G}}_k+a_K^{(j)}\widetilde{\bm{G}}_K\right)\bm{v}\le
	\bm{v}^T\left(\sum_{k=0}^Ka_k^{(j)}\bm{G}_k\right)\bm{v}\le \overline{c}\,
	\bm{v}^T\left(\sum_{k=0}^{K-1}{a}_k^{(j)}
	{\bm{G}}_k+a_K^{(j)}\widetilde{\bm{G}}_K\right)\bm{v}.
	\end{equation}
	From~\cref{eq:lemma_splitting} and \cref{rem:H_sign}, using the same technique as in \cref{th:truncated}, we get
	\begin{equation}
	\underline{c}\,\bm{v}^T\left(\vert a_K^{(j)}\vert\bm{G}_0+ \mu\,a_K^{(j)}\widetilde{\bm{G}}_K\right)\bm{v}\le\bm{v}^T\left(\vert a_K^{(j)}\vert\bm{G}_0+ \mu\, a_K^{(j)}\bm{G}_K\right)\bm{v}\le \overline{c}\,\bm{v}^T\left(\vert a_K^{(j)}\vert{\bm{G}}_0+ \mu\, a_K^{(j)}\widetilde{\bm{G}}_K\right)\bm{v}.
	\end{equation}
	Proceeding further as in \cref{eq:first_for_splitting}--\cref{eq:last_for_splitting} of the proof of \cref{th:truncated}, we obtain the desired inequality~\cref{eq:G_for_splitting}.
\end{proof}   

\begin{lemma}\label{th:splitting_C}
	Assume the complete polynomials. Under assumption \cref{eq:omega} and for $\bm{M}$ defined by \cref{eq:MinLemma6}, the constants $\underline{c}$ and $\overline{c}$ defined in~\cref{eq:const_beta} can be obtained as 
	\begin{equation}\label{eq:lemma_splitting3}
	\underline{c}=\min_{t=1,\dots,s}\lambda_{\text{min}}(\bm{H}_{t})\quad \text{and}\quad
	\overline{c}=\max_{t=1,\dots,s}\lambda_{\text{max}}(\bm{H}_{t}),
	\end{equation}
	where $\lambda_{\text{min}}(\bm{H}_{t})$ and $\lambda_{\text{max}}(\bm{H}_{t})$ are the smallest and the largest eigenvalue, respectively, of $\bm{H}_{t}$ defined in \cref{eq:mat_H}.
\end{lemma}
\begin{proof}
	Using \cref{th:1}, we only need to prove for every $\bm{v}\in\mathbb{R}^{N_{\rm P}}$
	\begin{equation}\label{eq:lemma6_1}
	\underline{c}\,\bm{v}^T\left(\sum_{k=0}^{K}
	a_k^{(j)}{\widetilde{\bm{G}}}_k\right)\bm{v}\le
	\bm{v}^T\left(\sum_{k=0}^{K}
	a_k^{(j)}\bm{G}_k\right)\bm{v}\le \overline{c}\,
	\bm{v}^T\left(\sum_{k=0}^{K}{a}_k^{(j)}
	\widetilde{{\bm{G}}}_k\right)\bm{v},\quad j=1,\dots,N_{\rm elem}.
	\end{equation}
	It suffices to show
	\begin{equation}\label{eq:lemma6_2}
	\underline{c}\,
	\bm{v}^T( \bm{G}_0\pm\mu\,\widetilde{\bm{G}}_k)\bm{v}
	\le \bm{v}^T( \bm{G}_0\pm\mu   \bm{G}_k)\bm{v}\le  
	\,\overline{c}
	\bm{v}^T( \bm{G}_0\pm\mu\,\widetilde{\bm{G}}_k)\bm{v}, \quad k=1,\dots,K,
	\end{equation}
	from which \cref{eq:lemma6_1} can be obtained analogously as in \cref{th:truediag,th:truncated,th:splitting_TP}. 
	
	Since now the matrices do not have the tensor product form, to prove \cref{eq:lemma6_2}, we cannot proceed in the same way as in the previous lemmas. 
	The constants $\underline{c}$ and $\overline{c}$ are obtained as the extreme eigenvalues of the generalized eigenvalue problem
	\begin{equation}\label{eq:lemma6_G0G1_2}
	({\bm I}_{N^{\rm C}_s}\pm \mu\,{\bm{G}}_k)\bm{v}=\lambda 
	(\bm{I}_{N^{\rm C}_s}\pm \mu\, \widetilde{\bm{G}}_k)\bm{v},
	\end{equation}
	where $N^{\rm C}_t=\genfrac(){0pt}{2}{K+t-1}{K}$
	denotes the size of the basis of $K$-variate complete polynomials of degree at most $t-1$.
	
	Assume first $k=1$.
	Let $\bm{G}_1^{\rm R}$ and $\widetilde{\bm{G}}_1^{\rm R}$
	be the matrices obtained from $\bm{G}_1$ and $\widetilde{\bm{G}}_1$, respectively, by reordering 
	their rows and columns in such manner that the corresponding basis $K$-variate orthonormal polynomials are ordered anti-lexicographically as, for example, in~\cref{eq:ex_TP_C_A} (instead of the ordering 
	according to their growing total degree, which is used so far).
	Then both $\bm{G}_1^{\rm R}$ and $\widetilde{\bm{G}}_1^{\rm R}$ become block-diagonal. The diagonal blocks of $\bm{G}_1^{\rm R}$
	are tridiagonal matrices
	$\bm{I}_t+\bm{G}_{t,1}$ of variable sizes $t\times t$,
	$t\in\{1,\dots,s\}$. The corresponding diagonal blocks 
	of $\widetilde{\bm{G}}_1^{\rm R}$ are equal to
	those of $\bm{G}_1^{\rm R}$ up to the last sub-
	and super-diagonal elements which are annihilated in
	$\widetilde{\bm{G}}_1^{\rm R}$.
	The number and sizes of the diagonal blocks depend on 
	$K$ and $s$. For $k\neq 1$, we analogously choose ordering where the $k$-th index changes the fastest.
	Thus, using \cref{rem:H_sign}, the eigenvalue problem~\cref{eq:lemma6_G0G1_2} reduces to a number of independent eigenvalue problems 
	with the matrices $\bm{H}_t$ of sizes $t=1,\dots,s$ defined in~\cref{eq:mat_H}.
\end{proof}

To obtain the actual spectral bounds $\underline{c}$ and $\overline{c}$ in \cref{th:splitting_TP} and in \cref{th:splitting_C}, we need to evaluate $d_s^{-1} = e_{s}^T(\bm{I}_{s}+\mu\bm{G}_{s,1})^{-1}e_{s}$. This can be done recursively, using \cref{eq:recursive_form}, but it may be advantageous to exploit the relation between Jacobi matrices and the Gauss-Christoffel quadrature, as we will describe in this part. Defining 
\begin{equation}
f(t) = \frac{1}{1+\mu t}\,, \ \ J_{s} = \begin{pmatrix}&&&1\\&&1&\\&\iddots&&\\1&&&\end{pmatrix}\,, \ \ \text{and} \quad \widehat{\bm{G}}_{s,1} = J_{s}\,\bm{G}_{s,1}\,J_{s}^{-1},
\end{equation}
we can rewrite
\begin{align}
d_s^{-1} =\; e_{s}^T(\bm{I}_{s}+\mu\bm{G}_{s,1})^{-1}e_{s} &= e_{s}^T\, f(\bm{G}_{s,1})\, e_{s}= e_{1}^T\,J_{s}\, f(\bm{G}_{s,1})\,J_{s}^{-1}\,e_{1} = e_{1}^T\, f(\widehat{\bm{G}}_{s,1})\,e_{1}.
\end{align}
Since $\widehat{\bm{G}}_{s,1}$ is again a Jacobi matrix, $e_{1}^T\, f(\widehat{\bm{G}}_{s,1})\,e_{1}$ can be computed as the Gauss--Christoffel quadrature of the integral of $f$, i.e.,
\begin{equation}\label{eq:quadrature_relations}
e_{1}^T\, f(\widehat{\bm{G}}_{s,1})\,e_{1} = \sum_{j=1}^{s} \widehat{\omega}_j^{(s)}f(\widehat{\lambda}_j^{(s)}) = \sum_{j=1}^{s} \frac{\widehat{\omega}_j^{(s)}}{1+\mu\widehat{\lambda}_j^{(s)}},
\end{equation}
where $\widehat{\lambda}_j^{(s)}$ and $\widehat{\omega}_j^{(s)}$ are the nodes and the weights, respectively, of the Gauss--Christoffel quadrature defined by $\widehat{\bm{G}}_{s,1}$; see \cite[Chap.~3]{Liesen2013Krylov} for a comprehensive overview of relations between Jacobi matrices, orthogonal polynomials, underlying distribution functions, and the Gauss--Christoffel quadrature. 

Due to symmetry of the considered distributions, the weights and nodes are also symmetric, and we can write
\begin{equation}\label{eq:quadrature_symmetry}
d_s^{-1} = \sum_{j=1}^{s} \frac{\widehat{\omega}_j^{(s)}}{1+\mu\widehat{\lambda}_j^{(s)}} = \sum_{j=1}^{s} \frac{\widehat{\omega}_j^{(s)}}{1-\mu^2(\widehat{\lambda}_j^{(s)})^2}.
\end{equation}
Since the weights of the Gauss--Christoffel quadrature are positive, we also obtain monotonic dependence on $\mu$, meaning that the conditioning of $\bm{M}^{-1}\bm{A}$ improves with decreasing $\mu$. However, for a given $\mu$, one can have decreasing as well as increasing behavior in $s$. In the following part, we provide more explicit expressions for the weights and nodes for the considered approximation polynomials.
It is well known that the nodes are the roots $\{\lambda_j^{(s)}\}$ of the highest-degree polynomial $\psi_{s}$. Moreover, the weights can be obtained from  $\psi_{s}$ as well; see \cite[p.~120]{Liesen2013Krylov}.

\paragraph{Gegenbauer polynomials} For the Gegenbauer polynomials, the weights are given by
\begin{equation}
\widehat{\omega}_j^{(s)} = \frac{2s+2\gamma-2}{s+2\gamma-1}\,\frac{1-(\lambda_j^{(s)})^2}{s},
\end{equation}
yielding
\begin{equation}\label{eq:kappa_splitting_gegenbauer}
d_s^{-1} = \frac{2s+2\gamma-2}{s+2\gamma-1}\,\frac{1}{s}\,\sum_{j=1}^{s} \frac{1-(\lambda_j^{(s)})^2}{1-\mu^2(\lambda_j^{(s)})^2},
\end{equation}
which for $\mu = 1$ simplifies to
\begin{equation}\label{eq:Gegenbauer_simple}
d_s^{-1} = \frac{2s+2\gamma-2}{s+2\gamma-1}.
\end{equation}
Substituting $\gamma = 1$ and $\gamma = \frac{1}{2}$, we obtain the spectral bounds of $\bm{H}_s$ for the Chebyshev and Legendre polynomials, respectively. 

\paragraph{Hermite polynomials} For the Hermite polynomials, the weights are given by
\begin{equation}
\widehat{\omega}_j^{(s)} = \frac{1}{s},
\end{equation}
yielding
\begin{equation}\label{eq:kappa_splitting_hermite}
d_s^{-1} = \frac{1}{s}\,\sum_{j=1}^{s} \frac{1}{1-\mu^2(\lambda_j^{(s)})^2}.
\end{equation}

\begin{remark}\label{rem413}
  \rev{For the splitting-based preconditioning, the constants $\underline{c}$ and $\overline{c}$ can be used to estimate the strengthened Cauchy-Bunyakowski-Schwarz inequality constant $\gamma_{\text{{\tiny CBS}}}\in[0,1)$, defined as the smallest $\gamma$ satisfying
  	\begin{equation}
  	({v}_U,{v}_W)_A\le\gamma({v}_U,{v}_U)_A({v}_W,{v}_W)_A,\qquad {v}_U\in U,\; {v}_W\in W;
  	\end{equation}
  	see, e.g.,~\cite{Axelsson1996Iterative,Bespalov2014Energy,KraMar09,Pultarova2015Adaptive}). In particular, it holds that
  	\begin{equation}
  		\gamma_{\text{{\tiny CBS}}}\leq\overline{c}-1 \  \, (\,= 1 - \underline{c}\,).
  	\end{equation}} 

\rev{The constant $\gamma_{\text{{\tiny CBS}}}$ can be used in the two-by-two block Gauss-Seidel preconditioning (also called the Schur-complement or multiplicative two-level preconditioning), with the resulting condition number 
 	bounded as
 	\begin{equation}
 	\kappa(\bm{M}_{\rm GS2}^{-1}\bm{A}) \leq \frac{1}{1-\gamma_{\text{{\tiny CBS}}}^2}  \ \, \left(\,\le \frac{1}{1-(\overline{c}-1)^2}= \frac{1}{d_t}\,\right);
 	\end{equation} 
 	see, e.g., \cite[Chapter~9]{Axelsson1996Iterative}, \cite[Sections~2.2 and~2.3]{KraMar09},
 	\cite{Sousedik2014Truncated,Sousedik2014Hierarchical}, and the examples in \cref{sec:num_exp} with the results summarized in Tables~\ref{tab:set4} and~\ref{tab:set5}.}

\rev{Further, if ${u}_U$ and ${u}_V$ represent the solutions of our problem in the spaces $U$ and in $V$, respectively, then the `solution improvement' $ {u}_V-{u}_U$
can be estimated as
\begin{equation}\label{eq:improvement_bound}
\Vert{e}_W\Vert_A^2\le\Vert {u}_V-{u}_U\Vert_A^2\le \frac{1}{1-\gamma_{\text{{\tiny CBS}}}^2}\Vert{e}_W\Vert_A^2.
\end{equation}
where ${e}_W$ is a solution of a certain problem restricted to the (small) space $W$;
see, e.g., \cite[Theorem~5.2]{AinsworthOden}.  Using \cref{eq:improvement_bound}, one can estimate $\Vert {u}_V-{u}_U\Vert_A$ in advance, without solving the (large) problem in $V$, which can be exploited in adaptivity.}

\rev{Due to Galerkin orthogonality, we have $\Vert u-u_U\Vert_A^2=\Vert u-u_V\Vert_A^2+
\Vert u_V-u_U\Vert_A^2$ where $u$ is the exact solution of~\eqref{eq:weak} in $\widetilde{V}$. Thus~\eqref{eq:improvement_bound} provides also an a~posteriori
lower bound of the energy norm of the error $u-u_U$ in $\widetilde{V}$; see, e.g., \cite{AinsworthOden,Bespalov2014Energy,Crowder2018}. 
}
\end{remark}

\subsection{Numerical examples}\label{sec:num_exp}

In this section, we illustrate on some simple numerical examples how to apply the introduced theoretical tool. \rev{First, we} consider a one-dimensional problem in $D=[0,1]$ with homogeneous Dirichlet boundary conditions, uniform mesh with $N_{\rm elem}=30$ elements and with the nodes $x_0=0$, \dots, $x_{N_{\rm elem}}=1$, $K=3$, uniform distributions of $\xi_1$, $\xi_2$, and $\xi_3$ with images in $[-1,1]$. The approximation and test spaces are spanned by a tensor product of continuous piece-wise linear functions defined
on $D$ and of $K$-variate Legendre polynomials. We define three different settings for $a_k({x})$, see \cref{tab:setting} and \cref{fig:func_a},
which are considered constant on every interval,
$a_k(x)=a_k(x_j^c)$ on $(x_j,x_{j+1})$
where $x_j^c=(x_j+x_{j+1})/2$, 
$j=0,\dots,N_{\rm elem}-1$.
We compute the corresponding $\mu$ and 
$\mu_{\rm class}$ defined in~\cref{eq:omega} and~\cref{eq:stronger_assum}, respectively, i.e.,
\begin{equation} \label{eq:mu_new_and_old}
\mu=\max_{j=1,\dots,N_{\rm elem}}\sum_{k=1}^3\vert a_k(x^c_j)\vert, \quad \mu_{\rm class}=\sum_{k=1}^3\max_{j=1,\dots,N_{\rm elem}}\vert a_k(x^c_j)\vert.
\end{equation}

\begin{figure}
	\captionsetup[subfigure]{labelformat=empty}
	\centering
	\begin{subfigure}[b]{0.3\textwidth}
		\includegraphics[width=\textwidth]{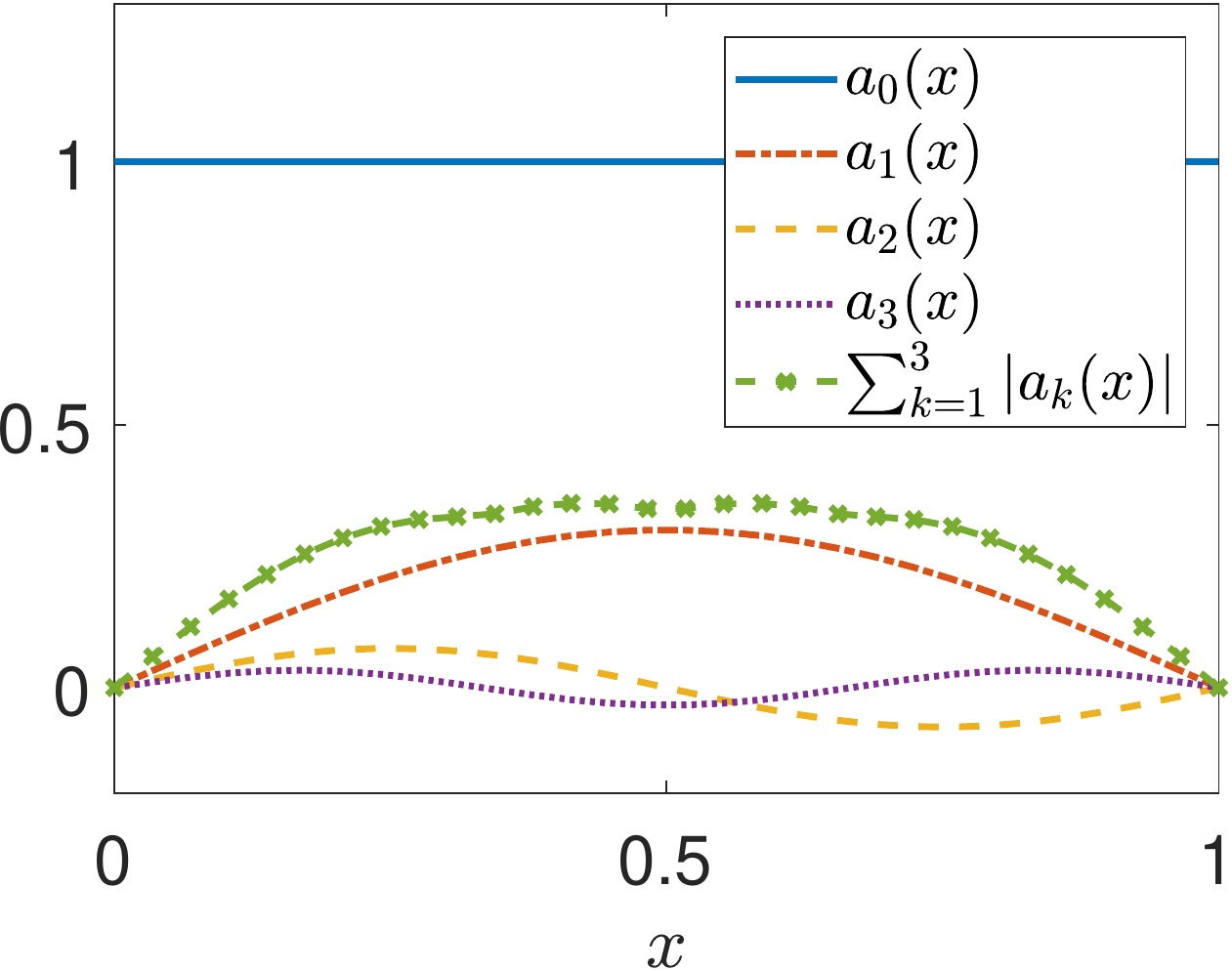}
		\caption{setting 1}
		\label{fig:1a}
	\end{subfigure}\quad
	\begin{subfigure}[b]{0.3\textwidth}
		\includegraphics[width=\textwidth]{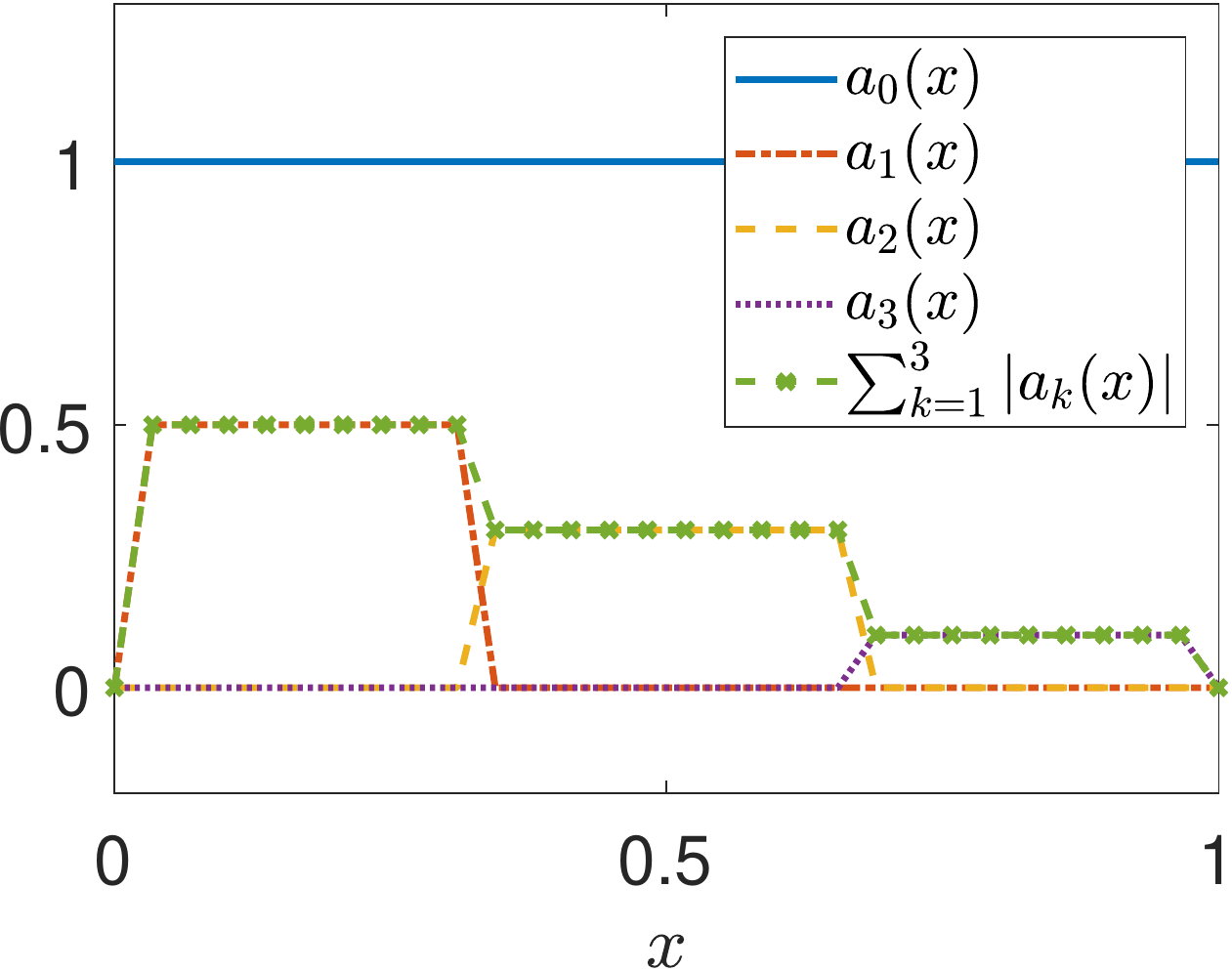}
		\caption{setting 2}
		\label{fig:1b}
	\end{subfigure}\quad
	\begin{subfigure}[b]{0.3\textwidth}
		\includegraphics[width=\textwidth]{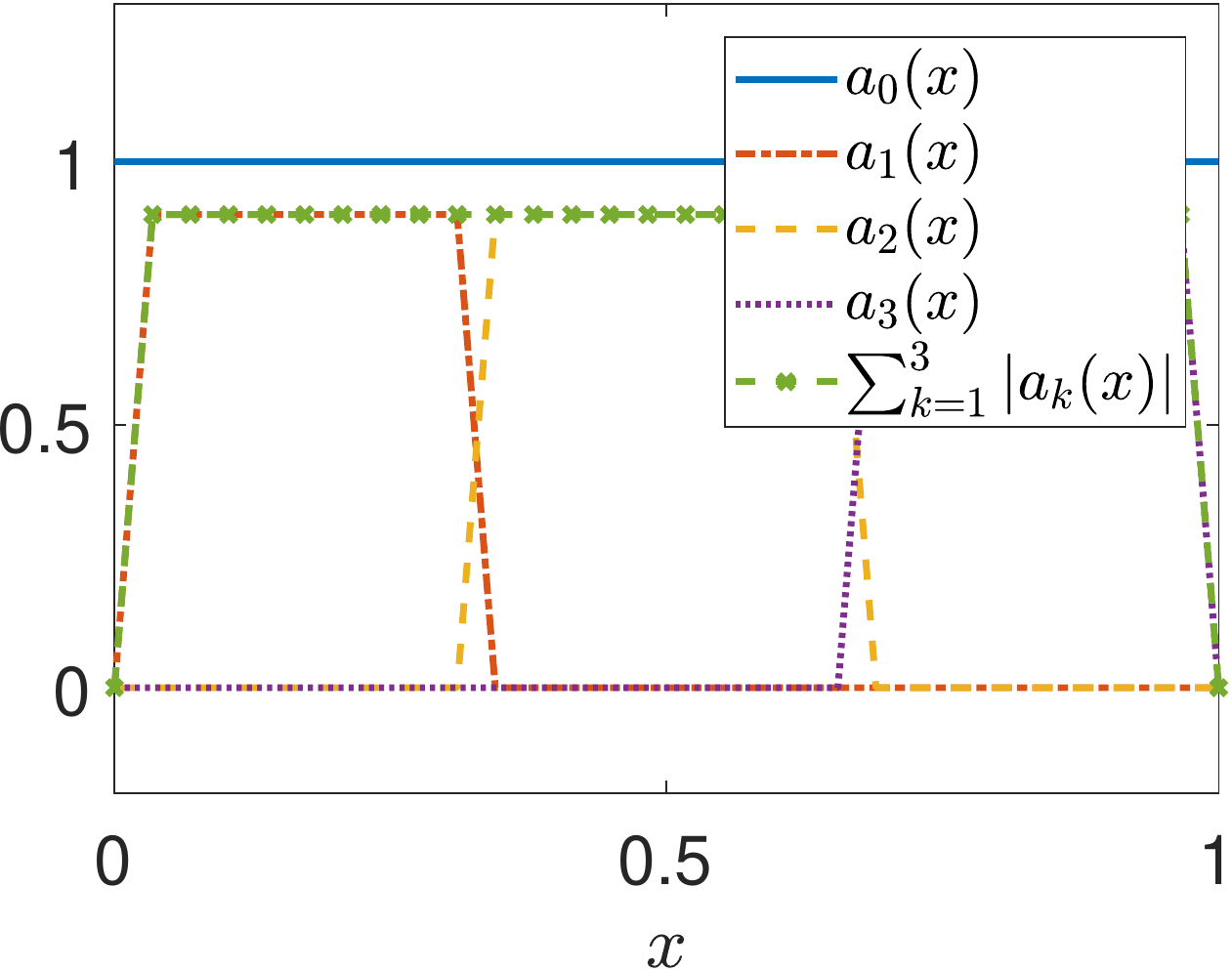}
		\caption{setting 3}
		\label{fig:1c}
	\end{subfigure}
	\caption{Functions $a_k(x)$, $k = 1,\ldots,K$, for the settings from \cref{tab:setting}.}\label{fig:func_a}
\end{figure}
\begin{table}
	\caption{Problem setting\rev{, one-dimensional problems.}}\label{tab:setting}
	\centering
	\bgroup\def\arraystretch{1.8}%
	{\footnotesize \begin{tabular}{c|cccc|cc}
		setting & $a_0({x})$ &$a_1({x})$& $a_2({x})$& $a_3({x})$&$\mu$&$\mu_\text{class}$\\\hline
		1& $1$&$\frac{0.3}{1^2}\sin(1\pi x)$ &
		$\frac{0.3}{2^2}\sin(2\pi x)$
		&$\frac{0.3}{3^2}\sin(3\pi x)$&0.35&0.41\\
		2& $1$& $0.5 \chi_{(0,1/3)}$& $0.3 \chi_{(1/3,2/3)}$& $0.1 \chi_{(2/3,1)}$&0.5&0.9\\
		3& $1$& $0.95 \chi_{(0,1/3)}$& $0.95 \chi_{(1/3,2/3)}$&
		$0.95 \chi_{(2/3,1)}$&0.95&2.85
	\end{tabular}}
	\egroup
\end{table}
\rev{We apply the mean-based preconditioning, see \cref{sec:mean-based}, to each of the settings.
We} compute the true extreme eigenvalues 
of the resulting preconditioned matrices, the theoretical spectral bounds $\underline{c}$ and $\overline{c}$ given by \cref{th:truediag} and \cref{cor1}, and 
the classical spectral bounds 
\begin{equation}
\underline{c}_{\rm class}=1-\mu_{\rm class} \lambda_{\rm max}(\psi_{s}),\quad
\overline{c}_{\rm class}=1+\mu_{\rm class} \lambda_{\rm max}(\psi_{s}).
\end{equation} 
derived, e.g.,~in~\cite[Theorem~3.8]{Powell2009Block}.
The roots of the Legendre orthogonal polynomials are obtained from~\cite{Lowan}. 
The results for \rev{complete} polynomials are summarized in \cref{tab:1}. For all considered settings, we obtained
\begin{equation}\label{eq:old}
\underline{c}_{\rm class}\le 
\underline{c}\le \lambda_{\rm min}(\bm{M}^{-1}\bm{A})\le \lambda_{\rm max}(\bm{M}^{-1}\bm{A})\le \overline{c}\le \overline{c}_{\rm class}.
\end{equation}
Moreover, for the third setting, the classical bounds do not provide any useful 
information because $\mu_{\rm class} \lambda_{\rm max}(\psi_{s})>1$,
and thus $\underline{c}_{\rm class}<0$.

\begin{table}
	\caption{Mean-based preconditioning
		and \rev{complete} polynomials: a comparison of the new and the classical spectral bounds, and the extreme eigenvalues of the  preconditioned matrix for the three different settings from \cref{tab:setting}.}
	\label{tab:1}
	\centering
	{\footnotesize \bgroup\def\arraystretch{1.2}%
	\begin{tabular}{c|c|c|cccccc|cc}
		&$s-1$&$\rev{\kappa(\bm{A})}$&$\underline{c}_{\rm class}$
		&$\underline{c}$&$\lambda_{\rm min}(\bm{M}^{-1}\bm{A})$&$\lambda_{\rm max}(\bm{M}^{-1}\bm{A})$&$\overline{c}$&
		$\overline{c}_{\rm class}$ & $\overline{c}/ \underline{c} $  &$\overline{c}_{\rm class}/ \underline{c}_{\rm class}$\\[.2em]
		\hline&&&&&&&\\[-.8em]
		\multirow{6}{*}{\rotatebox{90}{setting 1}}&
		1 & 458.42& 0.76  &  0.80 &   0.83  & 1.17 &  1.20  &  1.24 &   1.51 & 1.62  \\
		&2 & 498.47 &0.68  &  0.73 &   0.76   & 1.24 &   1.27 &   1.32  &   1.75  &1.92\\
		&\dots &&&&&&&&&\\
		&6 & 546.55 &0.61  &  0.67  &  0.69   & 1.31  &  1.33  &  1.39  &   2.00 & 2.26\\
		&7 & 550.80 &0.61  &  0.66  &  0.68  &  1.32  &  1.34  &  1.39  &  2.02 &  2.29 \\[.2em]
		\hline&&&&&&&\\[-.8em]
		\multirow{6}{*}{\rotatebox{90}{setting 2}}&
		1 & 542.75& 0.48  &  0.71  &  0.71   & 1.29  &  1.29 &   1.52  &   1.81  &  3.16\\
		&2 & 629.41 &0.30  &  0.61  &  0.61  &  1.39  &  1.39  &  1.70  &   2.26  &  5.60\\
		&\dots &&&&&&&&&\\
		&6 & 739.40  &0.15 &   0.53  &  0.53   & 1.47  &  1.47  &  1.85  &   2.81  & 12.72\\
		&7 & 749.57  &0.14  &  0.52  &  0.52 &   1.48 &   1.48 &   1.86  &  2.85  & 13.73\\[.2em]
		\hline&&&&&&&\\[-.8em]
		\multirow{6}{*}{\rotatebox{90}{setting 3}}&
		1 & 947.79 & -0.65  &   0.45 &   0.45 &   1.56 &   1.56 &   2.65 &    3.43 & - \\
		&2 & 1596.34  &-1.21  &   0.26 &   0.26  &  1.74&    1.74   &  3.21&     6.57  & - \\
		&\dots &&&&&&&&&\\
		&6 & 4576.93 & -1.71  &   0.10 &   0.10  &  1.90  & 1.90   &   3.71 &   19.34  & - \\
		&7 & 5294.63 & -1.74 &    0.09 &   0.09  &  1.91  &   1.91  &  3.74  &   21.80 &  - \\	
	\end{tabular}
	\egroup}
\end{table}

\rev{Next, we consider a two-dimensional problem in $D=[0,1]^2$ with homogeneous Dirichlet boundary conditions, uniform mesh with $N_{\rm elem}=20^2 = 400$ elements, uniform distributions of $\xi_k$, $k=1,\ldots,K$, with images in $[-1,1]$. The approximation and test spaces are spanned by a tensor product of continuous piece-wise bilinear functions defined
on $D$ and of $K$-variate Legendre polynomials. We define two different settings for $a_k({\bm{x}})$, $k=0,\ldots,K$, see \cref{tab:setting_2D}, which are analogously as before considered constant on every element.  
For both settings, we investigate the spectral bounds for the splitting-based preconditioning (SB) and the two-by-two block Gauss-Seidel preconditioning (GS2). The results are shown in \cref{tab:set4,tab:set5}. According to Remark~\ref{rem413}, the upper bound of $\kappa(\bm{M}_{\rm GS2}^{-1}\bm{A})$ is obtained as $1/d_t$. The values of $t$ used in Lemma~\ref{th:splitting_TP} for obtaining $\underline{c}$ and $\overline{c}$
are presented. 
\begin{table}
	\caption{Problem setting, two-dimensional problems.}\label{tab:setting_2D}
	\centering
	\bgroup\def\arraystretch{1.8}%
	{\footnotesize \begin{tabular}{c|cccc|c}
		setting & $a_0({x_1,x_2})$ &$a_1({x_1,x_2})$& $a_2({x_1,x_2})$& $a_3({x_1,x_2})$&$\mu$\\\hline
		4& $1$&${0.3}\sin(1\pi x_1)$ &
		${0.3}\sin(2\pi x_2)$
		&${0.3}\sin(2\pi x_1)$&0.83
	\end{tabular}}
	{\footnotesize \begin{tabular}{c|cccc|c}
		setting & $a_0({x_1,x_2})$ &$a_{2k+1}(x_1,x_2)$& $a_{2k+2}(x_1,x_2)$\\\hline
		5& $1$&$\frac{0.9}{K}\sin\left((k+1)\pi x_1\right)$&$\frac{0.9}{K}\sin\left((k+1)\pi x_2\right)$
\end{tabular}}
	\egroup
\end{table}
\vspace*{.4cm}
\begin{table}[ht]
		\caption{Splitting-based and two-by-two block Gauss-Seidel preconditioning for complete polynomials. Results for setting 4 (Table~\ref{tab:setting_2D}), 
		$K=3$, polynomial degree $s-1=1,\dots,5$.}\label{tab:set4}
		\centering
		\bgroup\def\arraystretch{1.2}
		{\footnotesize \begin{tabular}{c|c|cc|cc|c}
			$s-1$ &		$\kappa(\bm{A})$ &  	
			$\kappa(\bm{M}_{\rm SB}^{-1}\bm{A})$ &   
			$\overline{c}/\underline{c}$ 	
			& $\kappa(\bm{M}_{\rm GS2}^{-1}\bm{A})$ & $1/d_t$&$t$\\
			\hline
			1 &265.65&1.76&2.83& 1.08& 1.30& 2\\
			2 &334.62&2.13 &2.90&1.15& 1.31& 3\\
			3 &384.58&2.36&2.90&1.20&1.31&3  \\
			4 &420.15&2.50&2.90&1.22&1.31 & 3\\
			5 &446.06&2.56&2.90&1.24&1.31&3
	\end{tabular}}
	\egroup
\end{table}
\begin{table}[ht]
	\caption{Splitting-based and two-by-two block Gauss-Seidel preconditioning for complete polynomials. Results for setting 5 (Table~\ref{tab:setting_2D}),
	polynomial degree $s-1 = 2$,  expansion length $K=1,\dots,7$. }\label{tab:set5}
	\centering
	\bgroup\def\arraystretch{1.2}
	{\footnotesize \begin{tabular}{c|c|cc|cc|c|c}
		$K$ &		$\kappa(\bm{A})$ &  	
		$\kappa(\bm{M}_{\rm SB}^{-1}\bm{A})$ &   
		$\overline{c}/\underline{c}$ 	
		& $\kappa(\bm{M}_{\rm GS2}^{-1}\bm{A})$ & $1/d_t$&$t$&$\mu$\\
		\hline
		1& 580.00 & 3.36 & 3.38 & 1.41 & 1.42 & 3&0.90\\
		2 & 437.88  & 2.74 & 3.38 & 1.28 & 1.42 & 3& 0.90\\
		3 & 334.62 &  2.13 & 2.90  & 1.15 & 1.31  & 3 &0.83\\
		4 & 293.51  & 1.88  & 2.70 & 1.10 & 1.27  & 3 &0.79\\
		5 & 272.26  &  1.73 &  2.59 & 1.08 &  1.24  & 2&0.77 \\
		6 & 258.72  &  1.63 & 2.52  & 1.06 &  1.23 &  2&0.75\\
		7& 247.96 & 1.56& 2.48&1.05 & 1.22 & 2&0.74
	\end{tabular}}
	\egroup
\end{table}}

\section{Conclusion}\label{sec:conclusion}
We introduced a new tool for obtaining guaranteed and two-sided spectral bounds for discretized stochastic Galerkin problems preconditioned by a matrix with modified stochastic part from the spectral information of certain small matrices, from which the large stochastic Galerkin matrix is constructed. Moreover, this analysis only requires point-wise or local dominance of the deterministic part of the expansion of the parameter-dependent function $a(\bm{x},\bm{\xi})$, represented here by the parameter $\mu$, while the standard bounds are typically based on the absolute global dominance. The derived estimates \rev{are therefore also applicable} to problems where global dominance is not achieved. We showed for three types of block-diagonal preconditioners, including less standard ones, how this technique is used to obtain spectral bounds depending solely on $\mu$ and the properties of the stochastic approximation space (here classical orthogonal polynomials). From
the locality of $\mu$, it also follows that the obtained bounds are \rev{tighter than the classical ones.}
Similar ideas based on local properties of a preconditioner appear also in~\cite{GEMANIST2018}.

\section*{Acknowledgments}
The authors are grateful to Stefano Pozza for sharing his insight into the Gauss--Christoffel quadrature. \rev{We would also like to thank
anonymous referees for their comments and suggestions.}

\end{document}